\documentclass[a4paper,11pt]{article}

\usepackage[affil-it]{authblk}
\usepackage{amsmath}
\usepackage{amsfonts}
\usepackage{hyperref}

\usepackage{amsthm}
\usepackage{amssymb}
\usepackage{graphicx}
\usepackage{tensor}
\DeclareGraphicsExtensions{.jpg,.png,.pdf,.eps,.bmp,.gif}
\usepackage[latin1]{inputenc}
\newtheorem{theo}{Theorem}
\newtheorem{defi}{Definition}
\newtheorem{prop}[theo]{Proposition}
\newtheorem{lem}[theo]{Lemma}

\newtheorem{rem}[theo]{Remark}

\date{}

\title{From forward integrals to Wick-Itô integrals: the fractional Brownian motion and the Rosenblatt process cases}
\author{Benjamin Arras
\thanks{Electronic address: \texttt{arrasbenjamin@gmail.com}}}

\begin{document}
\maketitle

\begin{abstract}
In this paper, we combine Hida distribution theory and Sobolev-Watanabe-Kree spaces in order to study finely the link between forward integrals obtained by regularization and Wick-Itô integrals with respect to fractional Brownian motion and the Rosenblatt process. The new methodology developed in this paper allows to retrieve results for fractional Brownian motion and to obtain new results regarding the Rosenblatt process. In particular, an Itô formula for functionals of the Rosenblatt process is obtained which holds in the space of square-integrable random variables.
\end{abstract}

\section*{Introduction}
\subsection*{Context}
Since the construction of the Itô integral, there have been many approaches to extend stochastic integration for integrand and integrator processes which are not covered by the classical theory. Forward integration is a natural generalization of Itô integration allowing for anticipating integrands and for more general integrator processes. There are basically two approaches in order to define a forward integral with respect to Brownian motion: by means of Wiener analysis (and/or white noise analysis) as done in  \cite{BM82,NP88,KR88,AP91} and by regularization techniques first introduced in \cite{RV93}. The Brownian forward integral allows for anticipating integrands, is well approximated by forward Riemann sums and is linear when the integrand is a constant random variable.\\
\\
\noindent
The regularization technique developed by F. Russo and P. Vallois in \cite{RV93} is an almost pathwise method allowing to define forward, backward and symmetric integrals with respect to general integrators. This forward integral coincides essentially with the classical Itô integral when the integrator is a semi-martingale. Moreover, in the Brownian motion case, thanks to a Wiener analysis point of view, it is possible to link the forward integral defined by regularization techniques and the forward integral introduced thanks to Malliavin calculus tools (see Theorem 2.1 of \cite{RV93}). They are the same when the limiting procedure of the regularization is strengthened (see Remark 2.2 of \cite{RV93}).\\
\\
\noindent
Regarding fractional Brownian motion (fBm), the regularization techniques can be applied readily when $H>1/2$. Indeed, in this case, fBm is a zero quadratic variation process which admits a modification whose sample path are almost surely $\eta$-Hölder continuous, for any $1/2<\eta<H$. Therefore, the forward, backward and symmetric integrals with respect to fBm exist and essentially coincide (for regular enough integrands). In \cite{AN03}, the explicit link is made between the symmetric integral and the divergence operator associated with fractional Brownian motion for $H>1/2$. In this Gaussian framework, a term involving the stochastic derivative of the integrand appears. Moreover, when $H>1/2$, the Young type integral with respect to fractional Brownian motion can be defined and coincide with the forward integral with respect to fBm obtained by regularization (see Proposition $3$ page $155$ of \cite{RV07}). In this framework ($H>1/2$), there is therefore an unequivocal stochastic integral with respect to fBm defined by pathwise methods.\\
\\
Moreover, fBm belongs to the family of Hermite processes. These processes appear in non-central limit theorems (see e.g. \cite{DM79,T75,T79}). They are defined, for $d\geq 1$, by:
\begin{align*}
\forall t>0\quad X^{H,d}_t=c(H_0)\int_{\mathbb{R}}...\int_{\mathbb{R}}\left(\int_0^t\prod_{j=1}^d(s-x_j)^{H_0-1}_+ds\right)dB_{x_1}...dB_{x_d},
\end{align*}
where $\{B_x:\ x\in\mathbb{R}\}$ is a two-sided Brownian motion, $c(H_0)$ is a normalizing constant such that $\mathbb{E}[|X^{H,d}_1|^2]=1$ and $H_0=\frac{1}{2}+\frac{H-1}{d}$ with $H\in(\frac{1}{2},1)$. For $d=1$, one recovers fBm denoted by $\{B^H_t\}$ and for $d=2$, the process is named the Rosenblatt process, denoted by $\{X^H_t\}$ in the sequel. Hermite processes share many properties in common with fBm. Indeed, they are $H$-self-similar processes with stationary increments, have the same covariance structure and so their sample paths are almost-surely $\eta$-Hölder continuous, for every $1/2<\eta<H$. In particular, regarding their stochastic calculus, the regularization techniques of F. Russo and P. Vallois apply readily (as well as Young integration) so that the forward, backward, symmetric and Young integrals are well defined and coincide for smooth enough integrands.\\
\\
However, since their stochastic natures are very different (they live in different Wiener chaoses), one expects that the stochastic properties of the stochastic integrals with respect to Hermite processes of different orders would be quite different. This phenomenon does not seem reachable by purely pathwise integration methods and has been first observed partially in \cite{MR2374640} for the Rosenblatt process. Indeed, in \cite{MR2374640}, based on another representation of the Rosenblatt process, the author studies the link between the forward integral obtained by regularization techniques and the divergence integral with respect to the Rosenblatt process. He notes the appearance of two trace terms which differs significantly from the Gaussian case (see Theorem $2$ of Section $7$ in \cite{MR2374640}). However, existence of these two trace terms is not fully studied in \cite{MR2374640} even in the case where the integrand process is a smooth functional of the Rosenblatt process (see Theorem $3$ of Section $8$ in \cite{MR2374640}). Therefore, there is still room for improvements.\\
\\
In this paper, we study the link between forward integration by regularization techniques and Wick-Itô integration with respect to fractional Brownian motion and with respect to the Rosenblatt process. We obtain explicit decompositions of the forward integrals in both cases when the integrand processes are smooth functionals of the integrator processes (see Theorems \ref{Main1} and \ref{Main2} below). In particular, we obtain existence and explicit simple formulae for the two trace terms appearing in the Rosenblatt process case. The methodology we develop is based on Hida distribution theory and on Sobolev-Watanabe-Kree spaces. We comment briefly on it.
\begin{itemize}
\item The first step in our procedure is to compute the $S$-transform of $F(X^{H,d}_t)(X^{H,d}_{t+\epsilon}-X^{H,d}_t)$, for $\epsilon>0$, and to identify each terms thanks to Hida distribution theory. The formulae hold true in $(S)^*$, the Hida distributions space.  
\item Then, we prove that each term of the decomposition is a real random variable by using appropriate stochastic gradient operators (and their adjoints) naturally linked to the integrator process. The regularity of the integrand process plays a role in these representations.
\item Finally, we prove convergence in $(L^2)$, the space of square-integrable random variables, for each term appearing in the decompositions. This last step ensures that the forward integrals obtained coincide with the forward integral defined by regularization. 
\end{itemize}
This methodology is applied to the fBm and the Rosenblatt process cases ($d=1,2$). Nevertheless, it seems robust enough to possibly handle the following extensions:
\begin{itemize}
\item Any Hermite processes of any order $d$, the difficulty being the increasing number of terms appearing in the decomposition of $F(X^{H,d}_t)(X^{H,d}_{t+\epsilon}-X^{H,d}_t)$ to analyse.
\item Any self-similar processes with stationary increments with $H\in (1/2,1)$ represented by:
\begin{align*}
\forall t>0\quad Y^{H,d}_t=c(H_0)\int_{\mathbb{R}}...\int_{\mathbb{R}}\left(\int_0^tq_{H,d}(s-x_1,...,s-x_d)ds\right)dB_{x_1}...dB_{x_d},
\end{align*}
where $q$ is a symmetric function on $\mathbb{R}^d$ verifying appropriate conditions (see e.g. \cite{MO86}).
\end{itemize}
Due to the prominent roles of the fBm and of the Rosenblatt process, we only study these cases. The generalizations will be done in subsequent papers.

\subsection*{Main results and some notations}
Before stating the main results of this paper we introduce some notations. We denote by $I_d$ the multiple Wiener-Itô integral of order $d$. Moreover, for any random variable $X$ and any $k\geq 1$, we denote by $\kappa_k(X)$ the $k$-th cumulant of $X$ when it exists. We denote by $(\nabla^{H-\frac{1}{2}})^*$ the adjoint of the stochastic gradient operator naturally associated with fractional noise (see Propositions \ref{ExAdGrad1} and \ref{RepAdGrad1}). We denote by $(\nabla^{(2)})^*$ and by $(\nabla)^*$ the adjoints of the first and second order stochastic gradients associated with the white noise (see Proposition \ref{ConAdOp}). We denote by $F, F', F'', F^{(3)},...$ the functional and its derivatives. Finally, we define the forward integral by regularization as the following limit in probability (for $F$ smooth enough):
\begin{align*}
\underset{\epsilon\rightarrow 0^+}{\lim}\int_a^b F(X^{H,d}_t)\dfrac{(X^{H,d}_{(t+\epsilon)\wedge b}-X^{H,d}_t)}{\epsilon}dt\overset{\mathbb{P}}{=}\int_a^bF(X^{H,d}_t)d^{-}X^{H,d}_t
\end{align*}
\vspace{0.2cm}
\begin{theo}\label{Main1}
Let $(a,b)\subset\mathbb{R}_+$. Let $F$ be a continuously differentiable function on $\mathbb{R}$ such that:
\begin{align*}
\forall x\in\mathbb{R},\ \max\{F(x),F'(x)\}\leq Ce^{\lambda x^2},
\end{align*}
for some $C>0$ and $\lambda>0$ with $\lambda<1/(4b^{2H})$. Then, we have, in $\big(L^2\big)$:
\begin{align*}
\int_a^bF(B^H_t)d^{-}B^H_t=(\nabla^{H-\frac{1}{2}})^*(F(B^H_{.}))+H\int_a^bt^{2H-1}F'(B^H_t)dt.
\end{align*}
\end{theo}

\begin{rem}
\begin{itemize}
\item This result should be compared with Proposition $3$ of \cite{AN03} where a similar result holds true. The authors use the representation of the fBm as a Wiener integral on a compact interval and the intrinsic Malliavin calculus with respect to it whereas we use its representation as a Wiener integral on $\mathbb{R}$ and stochastic gradient operators on the white noise space.
\item As a straightforward corollary, we obtain the following well-known Itô formula for $F\in C^2(\mathbb{R})$ with appropriate growth conditions:
\begin{align*}
F(B^H_b)-F(B^H_a)=(\nabla^{H-\frac{1}{2}})^*(F'(B^H_{.}))+H\int_a^bt^{2H-1}F"(B^H_t)dt.
\end{align*}
\end{itemize}
\end{rem}

\begin{theo}\label{Main2}
Let $(a,b)\subset\mathbb{R}_+$. Let $F$ be an infinitely differentiable function with polynomial growth at most (together with its derivative). Then, we have, in $\big(L^2\big)$:
\begin{align*}
\int_a^b F(X^H_t)d^-X^H_t&=d(H)\big(\nabla^{(2)}\big)^*\bigg(\int_a^b F(X^H_t)\dfrac{(t-.)^{\frac{H}{2}-1}_+(t-\#)^{\frac{H}{2}-1}_+}{\Gamma(\frac{H}{2})^2}dt\bigg)+B(H)\nabla^*\bigg(\int_a^b F'(X^H_t)(t-.)^{\frac{H}{2}-1}_+\\
&\times I_1(l^H_{t,t})dt\bigg)+H\int_{a}^bt^{2H-1}F'(X^H_t)dt+\frac{H}{2}\kappa_3(X^H_1)\int_a^bt^{3H-1}F^{(2)}(X^H_t)dt\\
&+C(H)\int_a^bI_2(e^H_{t,t})F^{(2)}(X^H_t)dt,
\end{align*}
with,
\begin{align*}
l^H_{t,t}(x)&=\int_0^t(u-x)^{H/2-1}_+|t-u|^{H-1}du,\\
e^H_{t,t}(x_1,x_2)&=\int_0^t\int_0^t(u-x_1)^{\frac{H}{2}-1}_+(v-x_2)^{\frac{H}{2}-1}_+\mid u-t\mid^{H-1} \mid v-t\mid^{H-1}dudv,\\
B(H)&=\dfrac{4d(H)}{\big(\Gamma(\frac{H}{2})\big)^2}\sqrt{\dfrac{H(2H-1)}{2}},\\
C(H)&=\dfrac{2 d(H)}{\big(\Gamma(\frac{H}{2})\big)^2}H(2H-1).
\end{align*}
\end{theo}

\begin{rem}
\begin{itemize}
\item This result should be compared with Theorem $2$ of Section $8$ in \cite{MR2374640} where a similar result is obtained under the assumption of existences of the trace terms. The author use the representation of the Rosenblatt process on a compact interval and the Malliavin calculus with respect to Brownian motion whereas we use its representation as a double Wiener integral on $\mathbb{R}^2$ and stochastic gradient operators on the white noise space.
\item As a straightforward corollary, we have the following new Itô formula:
\begin{align*}
F(X^H_b)-F(X^H_a)&=d(H)\big(\nabla^{(2)}\big)^*\bigg(\int_a^b F'(X^H_t)\dfrac{(t-.)^{\frac{H}{2}-1}_+(t-\#)^{\frac{H}{2}-1}_+}{\Gamma(\frac{H}{2})^2}dt\bigg)\\
&+B(H)\nabla^*\bigg(\int_a^b F''(X^H_t)(t-.)^{\frac{H}{2}-1}_+I_1(l^H_{t,t})dt\bigg)+H\int_{a}^bt^{2H-1}F''(X^H_t)dt\\
&+\frac{H}{2}\kappa_3(X^H_1)\int_a^bt^{3H-1}F^{(3)}(X^H_t)dt+C(H)\int_a^bI_2(e^H_{t,t})F^{(3)}(X^H_t)dt.
\end{align*}
\item The previous Itô formula should be compared with Theorem $3.16$ of \cite{A15} where we obtain an Itô formula in the white noise sense for entire analytic functionals with growth conditions of the Rosenblatt process. The link between theses two formulae can be made by using iterated integration by parts on the white noise space thanks to the pointwise multiplications with $I_1(l^H_{t,t})$ and $I_2(e^H_{t,t})$ for smooth enough functionals.
\end{itemize}
\end{rem}

\subsection*{Organisation}
This paper is organized as follows. In the first section, we introduce the relevant tools from Hida distribution theory and we define the Sobolev-Watanabe-Kree spaces on the white noise space. In the second section, we define the fractional and the Rosenblatt noises, the stochastic integrals with respect to them and the associated stochastic gradient operators. In the third section, we start by analyzing the fractional brownian motion case and we end with the Rosenblatt process case. In particular, for the Rosenblatt process, we separate the studies of the $(L^2)$-convergences for each term appearing in the decomposition of $F(X^{H}_t)(X^{H}_{t+\epsilon}-X^{H}_t)$.


\section{Hida distribution and Sobolev-Watanabe-Kree spaces.}
In this section, we briefly remind the white noise analysis introduced by Hida and al. in \cite{HKPS93}. For a good introduction to the theory of white noise, we refer the reader to the book of Kuo \cite{K96}. The underlying probability space $(\Omega, \mathcal{F}, \mathbb{P})$ is the space of tempered distributions endowed with the $\sigma$-field generated by the open sets with respect to the weak* topology in $S'(\mathbb{R})$ and with the infinite dimensional Gaussian measure $\mu$ whose existence is ensured by the Bochner-Minlos theorem.\\
\\
For all $(\phi_1,...,\phi_n)\in S(\mathbb{R})$, the space of $C^{\infty}(\mathbb{R})$ functions with rapid decrease at infinity, the vector $(<.;\phi_1>,...,<.;\phi_n>)$ is a centered Gaussian random vectors with covariance matrix $(<\phi_i;\phi_j>)_{(i,j)}$. As it is written in Kuo \cite{K96}, for any function $f\in L^2(\mathbb{R})$, we can define $<.;f>$ as the random variable in $L^2(\Omega, \mathcal{F}, \mathbb{P})$ obtained by a classical approximation argument and the following isometry:
$$\forall (\phi,\psi)\in S(\mathbb{R})^2\quad\mathbb{E}[<;\psi><;\phi>]=<\psi;\phi>_{L^2(\mathbb{R})}$$
Thus, for any $t\in\mathbb{R}$, we define ($\mu$-almost everywhere):
\[B_t(.)=
\begin{cases}
<.;1_{[0;t]}> & t\geq 0 \\
-<.;1_{[t;0]}> & t<0
\end{cases}
\]
From the isometry property, it follows immediately that $B_t$ is a Brownian motion on the white noise space and, by the Kolmogorov-Centsov theorem, it admits a continuous modification. Moreover, using the approximation of any function $f\in L^2(\mathbb{R})$ by step functions, we obtain:
$$<;f>=\int_{\mathbb{R}}f(s)dB_s$$
We note $\mathcal{G}$, the sigma field generated by Brownian motion and $(L^2)=L^2(\Omega,\mathcal{G},\mathbb{P})$. By the Wiener-Itô theorem, any functionals $\Phi\in (L^2)$ can be expanded uniquely into a series of multiple Wiener-Itô integrals:
$$\Phi=\sum_{n=0}^{\infty}I_n(\phi_n)$$
where $\phi_n\in \hat{L}^2(\mathbb{R}^n)$, the space of square-summable symmetric functions. Using this theorem and the second quantization operator of the harmonic oscillator operator, $A=\frac{-d^2}{dx^2}+x^2+1$, Hida and al. introduced the stochastic space of test functions $(S)$ and its dual, the space of generalized functions $(S)^*$ or Hida distributions. We refer the reader to pages 18-20 of \cite{K96} for an explicit construction. We have the following Gel'fand triple:
$$(S)\subset (L^2)\subset (S)^*$$
We denote by $\left\langle \left\langle ;\right\rangle\right\rangle$ the duality bracket between elements of $(S)$ and $(S)^*$ which reduces to the classical inner product on $(L^2)$ for two elements in $(L^2)$.\\
In the context of white noise analysis, the main tool is the $S$-transform. It is a functional on $S(\mathbb{R})$ which characterizes completely the elements in $(S)^*$ (as well as the strong convergence in $(S)^*$). 

\begin{defi} 
Let $\Phi\in (S)^*$. For every function $\xi\in S(\mathbb{R})$, we define the $S$-transform of $\Phi$ by:
$$S(\Phi)(\xi)=\left\langle \left\langle \Phi; :\exp(<;\xi>):\right\rangle\right\rangle$$
where $:\exp(<;\xi>):=\exp(<;\xi>-\frac{||\xi||_{L^2(\mathbb{R})}^2}{2})=\sum_{n=0}^{\infty}\frac{I_n(\xi^{\otimes n})}{n!}\in (S)$.
\end{defi}

\begin{rem} 
For every $\Phi\in (L^2)$, we have:
$$S(\Phi)(\xi)=\mathbb{E}[\Phi :\exp(<;\xi>):]=\mathbb{E}^{\mu_{\xi}}[\Phi]$$
where $\mu_{\xi}$ is the translated infinite dimensional measure defined by:
$$\mu_{\xi}(dx)=\exp(<x;\xi>-\frac{||\xi||_{L^2(\mathbb{R})}^2}{2})\mu(dx).$$
\end{rem}
\noindent
Regarding the $S$-transform, we have the following properties and results:

\begin{theo}\label{STrans}
\begin{enumerate}
\item The $S$-transform is injective. If $\forall \xi\in S(\mathbb{R}), S(\Phi)(\xi)=S(\Psi)(\xi)$ then $\Phi=\Psi$ in $(S)^*$.
\item Let $\Psi\in (S)^*$ such that $\Psi=\sum_{n=0}^{\infty}I_n(\psi_n)$ with $\psi_n\in \hat{S}'(\mathbb{R}^n)$:
\begin{align*}
\forall\xi\in S(\mathbb{R}),\ S(\Psi)(\xi)=\sum_{n=0}^{\infty}\langle \psi_n;\xi^{\otimes n}\rangle.
\end{align*}
\item For $\Phi , \Psi\in (S)^*$ there is a unique element $\Phi\diamond\Psi\in (S)^*$ such that for all $\xi\in S(\mathbb{R}), S(\Psi)(\xi)S(\Phi)(\xi)=S(\Phi\diamond\Psi)(\xi)$. It is called the Wick product of $\Phi$ and $\Psi$.
\item Let $\Phi_n\in (S)^*$ and $F_n=S(\Phi_n)$. Then $\Phi_n$ converges strongly in $(S)^*$ if and only if the following conditions are satisfied:
\begin{itemize}
\item $\underset{n\rightarrow \infty}{\lim}F_n(\xi)$ exists for each $\xi\in S(\mathbb{R})$.
\item There exists strictly positive constants K, a and p independent of n such that:
$$\forall n\in\mathbb{N}, \forall\xi\in S(\mathbb{R})\quad|F_n(\xi)|\leq K\exp(a||A^p\xi||_{L^2(\mathbb{R})}^2)$$
\end{itemize}
\end{enumerate}
\end{theo}
\noindent
In the sequel, we introduce the differential calculus and Sobolev-Watanabe-Kree spaces on the white noise probability space. For further details, we refer the reader to chapter $9$ of \cite{K96} and chapter $5$ of \cite{HKPS93}. First, we define the Gâteaux derivative of elements in $(S)$ for direction in $S'(\mathbb{R})$.

\begin{theo}
Let $y\in S'(\mathbb{R})$ and $\Phi\in(S)$. The operator $D_y$ is continuous from $(S)$ into itself and we have:
\begin{align*}
\forall\omega\in S'(\mathbb{R})\quad D_y(\Phi)(\omega)=\sum_{n=1}^{\infty}nI_{n-1}(y\otimes_1\phi_n)(\omega),
\end{align*}
where we denote by $\otimes_1$ the contraction of order $1$ (see \cite{N06}).
\end{theo}

\begin{proof}
See Theorem $9.1$ of \cite{K96}. 
\end{proof}
\noindent
The next result states that every test random variable is actually infinitely often differentiable in Gâteaux and in Fréchet senses.

\begin{theo}\label{DevGF}
Let $\Phi\in (S)$. $\Phi$ is infinitely often Gâteaux differentiable in every direction of $S'(\mathbb{R})$ and infinitely often differentiable in Fréchet sense. Moreover, for every $k\in \mathbb{N}^*$ and for every $y_1,...,y_k\in (S'(\mathbb{R}))^k$, we have:
\begin{align*}
D_{y_1}\circ D_{y_2}\circ...\circ D_{y_k}(\Phi)=<y_1\otimes y_2\otimes...\otimes y_k; \nabla^{(k)}(\Phi)>,
\end{align*}
where $\nabla^{(k)}(\Phi)$ is the $k$-th Fréchet derivative of $\Phi$ and the equality stands in $(S)$. In particular, $\nabla^{(k)}(\Phi)\in \hat{S}(\mathbb{R}^k)\otimes (S)$.
\end{theo}

\begin{proof}
See Theorems $5.7$ and $5.14$ of \cite{HKPS93}.
\end{proof}
\noindent
We introduce as well the number operator, $N$, on $S$ in order to define Sobolev-Watanabe-Kree spaces.

\begin{defi} 
Let $r\geq 0$ and $\Phi\in (S)$ given by $\Phi=\sum_{n=0}^{\infty}I_n(\phi_n)$. We have:
\begin{align*}
N^{r}\Phi=\sum_{n=0}^{\infty}n^rI_n(\phi_n).
\end{align*}
Moreover, $N^r$ is a linear and continuous operator from $(S)$ into itself.\\
\end{defi}

\begin{proof}
See Theorem $9.23$ in \cite{K96}.
\end{proof}
\noindent
Let $r\geq 0$ and $(\mathcal{P})\subset (S)$ the algebra of polynomial random variables generated by elements of the form $I_1(\xi)$ with $\xi\in S(\mathbb{R})$. We denote by $(\mathcal{W}^{r,2})$ the completion of $(\mathcal{P})$ with respect to the norm:
\begin{align*}
\forall\Phi\in (\mathcal{P}),\ \|\Phi\|_{r,2}=\|(N+E)^{\frac{r}{2}}\Phi\|_{(L^2)},
\end{align*} 
where $E$ is the identity operator.

\begin{theo}\label{ConOp}
Let $k\in \mathbb{N}^*$ and $r\geq k$. $\nabla^{(k)}$ extends to a continous operator from $(\mathcal{W}^{r,2})$ into $\hat{L}^2(\mathbb{R}^k)\otimes (\mathcal{W}^{r-k,2})$.\\
\end{theo}

\begin{proof}
See Theorem $5.24$ and Corollary $5.25$ of \cite{HKPS93}.
\end{proof}
\noindent
We introduce a space of test random variables which is useful when considering functionals of fractional Brownian motion as well as functionals of the Rosenblatt process. For this purpose, we define, for any $r\in\mathbb{R}$ and any $p>1$, $(\mathcal{W}^{r,p})$ by the completion of $(\mathcal{P})$ with respect to the norm $\|(N+E)^{r/2}.\|_{(L^p)}$. We denote by $(\mathcal{W}^{\infty,\infty})$ the projective limit of the family $\{(\mathcal{W}^{r,p}),\ p>1,\ r\in\mathbb{R}\}$. Due to Meyer inequality (see chapter $1.5$ of \cite{N06}), for every $k\geq 1$, the operator $\nabla^{(k)}$ is continuous from $(\mathcal{W}^{\infty,\infty})$ into itself, this space is stable under pointwise multiplication and the following version of the product and chain rules hold.

\begin{prop}\label{ProdChain}
Let $\Phi,\Psi\in (\mathcal{W}^{\infty,\infty})$ and let $F$ be an infinitely continuously differentiable function on $\mathbb{R}$ such that $F$ and its derivatives have polynomial growth. Then, $F(\Phi)\in (\mathcal{W}^{\infty,\infty})$ and:
\begin{align*}
\nabla(\Phi\Psi)=\Phi\nabla(\Psi)+\Psi\nabla(\Phi),\\
\nabla(F(\Phi))=\nabla(\Phi)F'(\Phi).
\end{align*}
\end{prop}

\begin{proof}
This proposition is a consequence of the product and chain rules on $(\mathcal{P})$ for $\nabla$ as well as Remark $1$ page $78$, Proposition $1.5.1$ and Proposition $1.5.6$ of \cite{N06}.
\end{proof}
\noindent
We end this section by the definition of the adjoint of the Gâteaux derivative $D_y$ for every $y\in S'(\mathbb{R})$ and by continuity results regarding the adjoint operators $\nabla^*$ and $\big(\nabla^{(2)}\big)^*$.

\begin{theo} 
Let $y\in S'(\mathbb{R})$ and $\Psi\in(S)^*$. The adjoint operator $D^*_y$ is continuous from $(S)^*$ into itself and we have:
\begin{align*}
\forall\xi\in S(\mathbb{R})\quad S(D^*_y(\Psi))(\xi)=<y;\xi>S(\Psi)(\xi)=S(I_1(y)\diamond\Psi)(\xi)
\end{align*}
where $I_1(y)$ is a generalized Wiener-Itô integral in $(S)^*$. Moreover, we have the following generalized Wiener-Itô decomposition for $D^*_y(\Psi)$:
\begin{align*}
D^*_y(\Psi)(.)=\sum_{n=0}^{\infty}I_{n+1}(y\hat{\otimes}\psi_n).
\end{align*}
\end{theo}

\begin{proof}
See Theorem $9.12$, $9.13$ and the remark following corollary 9.14 in \cite{K96}.
\end{proof}

\begin{prop}\label{ConAdOp}
Let $r\geq 1$ and $k\geq 2$. Then, $\nabla^*$ is a continuous and linear operator from $L^2\big(\mathbb{R}\big)\otimes\big(\mathcal{W}^{r,2}\big)$ into $\big(\mathcal{W}^{r-1,2}\big)$ and $\big(\nabla^{(2)}\big)^*$ is a continuous and linear operator from $\hat{L}^2\big(\mathbb{R}^2\big)\otimes\big(\mathcal{W}^{k,2}\big)$ into $\big(\mathcal{W}^{k-2,2}\big)$.
\end{prop}

\begin{proof}
The first part of the proposition comes from Theorem $5.27$ of \cite{HKPS93}. Let us prove the second part. Let $k\geq 2$. Let $X\in \hat{L}^2\big(\mathbb{R}^2\big)\otimes\big(\mathcal{P}\big)$. We have, by duality:
\begin{align*}
\mid\mid \big(\nabla^{(2)}\big)^*\big(X\big)\mid\mid_{k-2,2}&=\mid\mid \big(E+N\big)^{\frac{k-2}{2}}\big(\nabla^{(2)}\big)^*\big(X\big)\mid\mid_{(L^2)},\\
&=\underset{\mid\mid\Phi\mid\mid_{(L^2)}=1}{\sup}\mid \langle \Phi; \big(E+N\big)^{\frac{k-2}{2}}\big(\nabla^{(2)}\big)^*\big(X\big)\rangle \mid,\\
&=\underset{\mid\mid\Phi\mid\mid_{(L^2)}=1}{\sup}\mid \langle \nabla^{(2)}\big(\big(E+N\big)^{\frac{k-2}{2}}\Phi\big);X \rangle_{\hat{L}^2(\mathbb{R}^2)\otimes(L^2)} \mid,\\
&=\underset{\mid\mid\Phi\mid\mid_{(L^2)}=1}{\sup}\mid \langle \big(3E+N\big)^{\frac{k}{2}}\nabla^{(2)}\big(\big(E+N\big)^{-1}\Phi\big);X \rangle_{\hat{L}^2(\mathbb{R}^2)\otimes(L^2)} \mid,\\
&\leq \underset{\mid\mid\Phi\mid\mid_{(L^2)}=1}{\sup} \mid\mid \nabla^{(2)}\big(\big(E+N\big)^{-1}\Phi\big) \mid\mid_{\hat{L}^2(\mathbb{R}^2)\otimes(L^2)}  \mid\mid \big(3E+N\big)^{\frac{k}{2}}X\mid\mid_{\hat{L}^2(\mathbb{R}^2)\otimes(L^2)},\\
&\leq \mid\mid \big(3E+N\big)^{\frac{k}{2}}X\mid\mid_{\hat{L}^2(\mathbb{R}^2)\otimes(L^2)},
\end{align*}
since, by continuity,
\begin{align*}
 \underset{\mid\mid\Phi\mid\mid_{(L^2)}=1}{\sup} \mid\mid \nabla^{(2)}\big(\big(E+N\big)^{-1}\Phi\big) \mid\mid_{\hat{L}^2(\mathbb{R}^2)\otimes(L^2)}\leq 1.
\end{align*}
Thus, we have:
\begin{align*}
\mid\mid \big(\nabla^{(2)}\big)^*\big(X\big)\mid\mid_{k-2,2}\leq 3^{\frac{k}{2}}\mid\mid \big(E+N\big)^{\frac{k}{2}}X\mid\mid_{\hat{L}^2(\mathbb{R}^2)\otimes(L^2)},
\end{align*}
which concludes the proof.
\end{proof}


\section{Stochastic analysis of fractional Brownian motion and of the Rosenblatt process.}
In this section, we state the definition of fractional Brownian motion and of the Rosenblatt process. Following \cite{B03} and \cite{A15}, we remind that these processes are $(S)^*$-differentiable and compute their $(S)^*$ derivatives. Then, we define stochastic derivative operators of first and second orders which play a significant role in the trace terms appearing in the relationship between Wick-Itô integral and forward integral with respect to these two processes. Moreover, we compute explicitely the Hilbert space adjoint of the first order stochastic gradient which is linked to Wick-Itô integral with respect to fractional Brownian motion. Therefore, we give a brief introduction to the stochastic integrals with respect to the fractional and the Rosenblatt noises in the Wick-Itô sense and make explicit the aforementioned link with the adjoint operator. In the rest of the article, we fix an interval $(a,b)$ included in $\mathbb{R}_+$.

\begin{defi} For $H>1/2$, we define fractional Brownian motion and the Rosenblatt process by:
\begin{align*}
\forall t\in (a,b),\ B^H_t=A(H)\int_{\mathbb{R}}\Big(\int_0^t\dfrac{(s-x)^{H-\frac{3}{2}}_+}{\Gamma(H-\frac{1}{2})}ds\Big)dB_x,\\
X^{H}_t=d(H)\int_{\mathbb{R}^2}\left(\int_{0}^{t}\dfrac{(s-x_1)^{\frac{H}{2}-1}_{+}}{\Gamma(\frac{H}{2})}\dfrac{(s-x_2)^{\frac{H}{2}-1}_{+}}{\Gamma(\frac{H}{2})}ds\right)dB_{x_1}dB_{x_2},
\end{align*}
where $A(H)$ and $d(H)$ are positive constants such that $\mathbb{E}[|B^H_1|^2]=\mathbb{E}[|X^H_1|^2]=1$ and defined by:
\begin{align*}
A(H)=\Big(\dfrac{\Gamma(H-\frac{1}{2})H(2H-1)\Gamma(\frac{3}{2}-H)}{\Gamma(2-2H)}\Big)^{\frac{1}{2}},\\
d(H)=\sqrt{\dfrac{H(2H-1)}{2}}\dfrac{(\Gamma(\frac{H}{2}))^{2}}{\beta(1-H;\frac{H}{2})}.
\end{align*}
\end{defi}

\begin{defi}
Fractional Brownian motion and the Rosenblatt process are $(S)^*$-differentiable and their derivatives, the fractional noise, $\{\dot{B}^H_t\}$, and the Rosenblatt noise, $\{\dot{X}^H_t\}$, admit the following $(S)$-transforms:
\begin{align*}
\forall\xi\in S(\mathbb{R}),\ S(\dot{B}^H_t)(\xi)=A(H)I^{H-\frac{1}{2}}_{+}(\xi)(t),\\
S(\dot{X}^H_t)(\xi)=d(H)(I^{\frac{H}{2}}_+(\xi)(t))^2.
\end{align*}
where for every $0<\alpha<1$, $I^{\alpha}_+(\xi)(t)=1/(\Gamma(\alpha))\int_{\mathbb{R}}(t-s)^{\alpha-1}_{+}\xi(s)ds$ is the fractional integral of order $\alpha$ of $\xi$ on the real line (\cite{SKM93} chapter $2$).
\end{defi}
\begin{proof}
See the proof of Lemma $2.15$, Theorem $2.17$ and Definition $2.18$ of \cite{B03} for fractional Brownian motion and Lemma $3.4$ of \cite{A15} for the Rosenblatt process.
\end{proof}
\noindent
The next Lemma is a technical one allowing us to define the first order stochastic derivative operators associated with fractional Brownian motion and the Rosenblatt process.

\begin{lem}\label{ConI1}
For every $\alpha\in (0,1/2)$ and every $r\geq 0$, the operator $I^{\alpha}_+\otimes E$ admits a continuous extension from $L^2(\mathbb{R})\otimes (\mathcal{W}^{r,2})$ to $L^2((a,b))\otimes (\mathcal{W}^{r,2})$.
\end{lem}

\begin{proof}
We define $I^{\alpha}_+\otimes E$ on simple element of $L^2(\mathbb{R})\otimes (\mathcal{W}^{r,2})$ by:
\begin{align*}
\forall \phi,\Phi\in L^2(\mathbb{R})\times (\mathcal{W}^{r,2}),\ (I^{\alpha}_+\otimes E)(\phi\otimes\Phi)=I^{\alpha}_+(\phi)\otimes\Phi.
\end{align*}
and we extend it by linearity. Moreover, we have:
\begin{align*}
\|I^{\alpha}_+(\phi)\otimes\Phi\|_{L^2((a,b))\otimes (\mathcal{W}^{r,2})}&=\|I^{\alpha}_+(\phi)\|_{L^2((a,b))}\|\Phi\|_{(\mathcal{W}^{r,2})},\\
&\leq C_{a,b,\alpha}\|I^{\alpha}_+(\phi)\|_{L^{\frac{2}{1-2\alpha}}(\mathbb{R})}\|\Phi\|_{(\mathcal{W}^{r,2})},\\
&\leq C_{a,b,\alpha}\|\phi\|_{L^2(\mathbb{R})}\|\Phi\|_{(\mathcal{W}^{r,2})}.
\end{align*}
since $I^{\alpha}_+$ is a continuous operator from $L^2(\mathbb{R})$ to $L^{\frac{2}{1-2\alpha}}(\mathbb{R})$ and $2/(1-2\alpha)>2$ (Theorem $5.3$ of \cite{SKM93}).
\end{proof}
\noindent
Consequently, we have the following result:

\begin{prop}\label{ExGrad1}
Let $r\geq 1$, $\alpha\in (0;1/2)$. There exists a continuous operator, denoted $\nabla^{\alpha}$, from $(\mathcal{W}^{r,2})$ into $L^2((a,b))\otimes (\mathcal{W}^{r-1,2})$ such that:
\begin{align*}
\forall \Phi\in(\mathcal{W}^{r,2}),\ \lambda\otimes\mu-a.e. (t,\omega)\in(a,b)\times S'(\mathbb{R}),\ \nabla^{\alpha}(\Phi)(t,\omega)=\sum_{n=1}^{\infty}nI_{n-1}(<\delta_t\circ I^{\alpha}_+;\phi_n>)(\omega),
\end{align*}
with $\Phi=\sum_{n=1}^{\infty}I_n(\phi_n)$.
\end{prop}

\begin{proof} 
From Lemma \ref{ConI1} and Theorem \ref{ConOp}, the operator $\nabla^{\alpha}=(I^{\alpha}_+\otimes E)\circ\nabla$ is continuous from $(\mathcal{W}^{r,2})$ into $L^2((a,b))\otimes (\mathcal{W}^{r-1,2})$. We only have to prove that the previous equality holds. First of all, notice that, for $n\geq 1$, by Theorem $24.1$ of \cite{SKM93}:
\begin{align*}
\big(\int_a^b\|<\delta_t\circ I^{\alpha}_+;\phi_n>\|^2_{L^2(\mathbb{R}^{n-1})}dt\big)^{\frac{1}{2}}&\leq C_{\alpha,a,b}\|I^{(0,...,0,\alpha)}_{+,...,+}(\phi_n)\|_{L^{(2,...,2,\frac{2}{1-2\alpha})}(\mathbb{R}^n)},\\
&\leq C_{\alpha,a,b} \|\phi_n\|_{L^2(\mathbb{R}^n)}<+\infty.
\end{align*}
Then, $I_{n-1}(<\delta_{(.)}\circ I^{\alpha}_+;\phi_n>)(.)$ is an element of $L^2((a,b))\otimes (\mathcal{W}^{r-1,2})$. By a standard argument, one can show that $\sum_{n=1}^NnI_{n-1}(<\delta_{(.)}\circ I^{\alpha}_+;\phi_n>)(.)$ converges in $L^2((a,b))\otimes (\mathcal{W}^{r-1,2})$, since $\Phi\in (\mathcal{W}^{r,2})$, to an element which we denote by $\sum_{n=1}^{\infty}nI_{n-1}(<\delta_{(.)}\circ I^{\alpha}_+;\phi_n>)(.)$. Since $(S)$ is dense in $(\mathcal{W}^{r,2})$, there exists a sequence $(\Phi_n)\in (S)^{\mathbb{N}}$ such that $\Phi_n\underset{n\rightarrow+\infty}{\rightarrow}\Phi$ in $(\mathcal{W}^{r,2})$. By continuity, $\nabla^{\alpha}(\Phi_n)\underset{n\rightarrow+\infty}{\rightarrow}\nabla^{\alpha}(\Phi)$ in $L^2((a,b))\otimes (\mathcal{W}^{r-1,2})$. Moreover, for all $n\in\mathbb{N}$, we have:
\begin{align*}
\forall (t,\omega)\in\mathbb{R}\times S'(\mathbb{R}),\ \nabla^{\alpha}(\Phi_n)(t,\omega)&=((I^{\alpha}_+\otimes E)\circ\nabla)(\Phi_n)(t,\omega),\\
&=\int_{-\infty}^{+\infty}\dfrac{(t-s)^{\alpha-1}_{+}}{\Gamma(\alpha)}\nabla(\Phi_n)(s,\omega)ds,\\
&=I^{\alpha}_+(\nabla(\Phi_n)(.,\omega))(t),\\
&=<\delta_t;I^{\alpha}_+(\nabla(\Phi_n)(.,\omega))>,\\
&=<\delta_t\circ I^{\alpha}_+;\nabla(\Phi_n)(.,\omega)>,\\
&=D_{\delta_t\circ I^{\alpha}_+}(\Phi_n)(\omega),\\
&=\sum_{m=1}^{\infty}mI_{m-1}(<\delta_t\circ I^{\alpha}_+;\phi^n_m>)(\omega).
\end{align*}
where we use Theorem \ref{DevGF} for the next to last equality. Since $\Phi_n\underset{n\rightarrow+\infty}{\rightarrow}\Phi$ in $(\mathcal{W}^{r,2})$, one can show that $\sum_{m=1}^{\infty}mI_{m-1}(<\delta_{(.)}\circ I^{\alpha}_+;\phi^n_m>)(.)$ converges to $\sum_{m=1}^{\infty}mI_{m-1}(<\delta_{(.)}\circ I^{\alpha}_+;\phi_m>)(.)$ in $L^2((a,b))\otimes (\mathcal{W}^{r-1,2})$. Thus, the equality holds in $L^2((a,b))\otimes (\mathcal{W}^{r-1,2})$.
\end{proof}
\noindent
The next Lemma is a technical one needed to define the second order stochastic derivative operator related to the Rosenblatt process.

\begin{lem} 
For every $r\geq 0$, The operator $I^{H/2}_+\otimes I^{H/2}_+\otimes E$ admits a continuous extension from $L^2(\mathbb{R}^2)\otimes (\mathcal{W}^{r,2})$ to $L^2((a,b)\times (a,b))\otimes (\mathcal{W}^{r,2})$.
\end{lem}

\begin{proof}
The proof is similar to one of Lemma \ref{ConI1}. $\Box$
\end{proof}
\noindent
Consequently, we have the following result:

\begin{prop}\label{ExGrad2}
Let $r\geq 2$. There exists a continuous operator, denoted $\nabla^{(2),H/2}$, from $(\mathcal{W}^{r,2})$ to $L^2((a,b)\times (a,b))\otimes (\mathcal{W}^{r-2,2})$ such that, for every $\Phi\in(\mathcal{W}^{r,2})$:
\begin{align*}
\lambda^{\otimes 2}\otimes\mu-a.e. (s,t,\omega)\in(a,b)^2\times S'(\mathbb{R}),\ \nabla^{(2),\frac{H}{2}}(\Phi)(s,t,\omega)=\sum_{n=2}^{\infty}n(n-1)I_{n-2}(<\delta_s\circ I_+^{\frac{H}{2}}\otimes\delta_t\circ I^{\frac{H}{2}}_+;\phi_n>)(\omega),
\end{align*}
with $\Phi=\sum_{n=0}^{\infty}I_n(\phi_n)$.
\end{prop}

\begin{proof} 
The operator $\nabla^{(2),H/2}=I^{\frac{H}{2}}_+\otimes I^{\frac{H}{2}}_+\otimes E\circ \nabla^{(2)}$ is a continuous operator from $(\mathcal{W}^{r,2})$ to $L^2((a,b)\times (a,b);(\mathcal{W}^{r-2,2}))$ by the previous Lemma and Theorem \ref{ConOp}. The equality is proved similarly to the one of Proposition \ref{ExGrad1} noting that, for $n\geq 2$:
\begin{align*}
\big(\int_{(a,b)\times(a,b)}\|<\delta_s\circ I_+^{\frac{H}{2}}\otimes\delta_t\circ I^{\frac{H}{2}}_+;\phi_n>\|^2_{L^2(\mathbb{R}^{n-2})}dsdt\big)^{\frac{1}{2}}&\leq C_{H,a,b}\|I^{(0,...,0,\frac{H}{2},\frac{H}{2})}_{+,..,+}(\phi_n)\|_{L^{(2,...,2,\frac{2}{1-H},\frac{2}{1-H})}(\mathbb{R}^n)},\\
&\leq C'_{H,a,b}\|\phi_n\|_{L^2(\mathbb{R}^n)}<\infty.
\end{align*}
allows to define an element of $L^2((a,b)\times (a,b);(\mathcal{W}^{r-2,2}))$ as $I_{n-2}(<\delta_{(.)}\circ I_+^{\frac{H}{2}}\otimes\delta_{(.)}\circ I^{\frac{H}{2}}_+;\phi_n>)(.)$. 
\end{proof}
\noindent
Next, we briefly remind the definitions of the fractional noise and of the Rosenblatt noise integrals for stochastic integrand processes which are $(S)^*$-integrable on $(a,b)$.

\begin{prop}
Let $\{\Phi_t;t\in (a,b)\}$ be a $(S)^*$ stochastic process such that:
\begin{enumerate}
\item $\forall\xi\in S(\mathbb{R})$, $S(\Phi_.)(\xi)$ is measurable on $(a,b)$.
\item There is a $p\in\mathbb{N}$, a strictly positive constant $a$ and a non-negative function $L\in L^1((a,b))$ such that:
\begin{align*}
\forall\xi\in S(\mathbb{R}),\ |S(\Phi_t)(\xi)|\leq L(t)\exp\left(a||A^p\xi||^2_2\right)
\end{align*}
\end{enumerate}
Then, $\Phi_t\diamond \dot{B}^H_t$ and $\Phi_t\diamond \dot{X}^H_t$ are $(S)^*$ integrable over $(a,b)$ and we define the fractional noise integral and the Rosenblatt noise integral of $\{\Phi_t\}$ by:
\begin{align*}
\int_{(a,b)}\Phi_tdB^H_t=\int_{(a,b)}\Phi_t\diamond \dot{B}^H_tdt,\\
\int_{(a,b)}\Phi_tdX^H_t=\int_{(a,b)}\Phi_t\diamond \dot{X}^H_tdt.
\end{align*}
Moreover, we have the following representation:
\begin{align*}
\int_{(a,b)}\Phi_tdB^H_t=\int_{(a,b)} (D^*_{A(H)\delta_t\circ I^{H-\frac{1}{2}}_+})(\Phi_t)dt,\\
\int_{(a,b)}\Phi_tdX^H_t=\int_{(a,b)} (D^*_{\sqrt{d(H)}\delta_t\circ I^{\frac{H}{2}}_+})^2(\Phi_t)dt.
\end{align*}
\end{prop}

\begin{proof}
See Definition-Theorem $3.10$ of \cite{A15} for the Rosenblatt noise integral and Section $3.3$ of \cite{B03} for the fractional noise integral.
\end{proof}
\noindent
Finally, we compute the Hilbert space adjoint of $\nabla^{\alpha}$, for $\alpha\in (0;1/2)$, and link this operator with the fractional noise integral for a certain class of stochastic integrand processes.

\begin{prop}\label{ExAdGrad1}
Let $\alpha\in (0,1/2)$ and $r\geq 1$. The operator $(\nabla^{\alpha})^*=\nabla^*\circ (I^{\alpha}_{-}(\mathbb{I}_{(a,b)}.)\otimes E)$ is a linear and continuous operator from $L^2((a,b))\otimes (\mathcal{W}^{r,2})$ into $(\mathcal{W}^{r-1,2})$, where $I^{\alpha}_{-}$ is defined by:
\begin{align*}
\forall\xi\in S(\mathbb{R}),\ I^{\alpha}_{-}(\xi)(t)=\dfrac{1}{\Gamma(\alpha)}\int_{\mathbb{R}}(s-t)^{\alpha-1}_{+}\xi(s)ds.
\end{align*}
\end{prop}

\begin{proof}
 By Proposition \ref{ExGrad1}, $\nabla^{\alpha}=(I^{\alpha}_{+}\otimes E)\circ \nabla$ is a linear and continuous operator from $(\mathcal{W}^{r,2})$ into $L^2((a,b))\otimes (\mathcal{W}^{r-1,2})$. Thus, by definition, $(\nabla^{\alpha})^*$ is a linear and continuous operator from $L^2((a,b))\otimes (\mathcal{W}^{r-1,2})^*$ into $(\mathcal{W}^{r,2})^*$. Moreover, $(\nabla^{\alpha})^*$ is equal to $\nabla^*\circ (I^{\alpha}_{+}\otimes E)^*$. Let us compute $(I^{\alpha}_{+}\otimes E)^*$. We have, for every $s\geq 0$, $f\in L^2((a,b))\otimes(\mathcal{W}^{s,2})^*$ and $g\in L^2(\mathbb{R})\otimes(\mathcal{W}^{s,2})$:
\begin{align*}
<(I^{\alpha}_{+}\otimes E)^*(f);g>_{(L^2(\mathbb{R})\otimes(\mathcal{W}^{s,2})^*,L^2(\mathbb{R})\otimes(\mathcal{W}^{s,2}))}=<f,(I^{\alpha}_{+}\otimes E)(g)>_{(L^2((a,b))\otimes(\mathcal{W}^{s,2})^*,L^2((a,b))\otimes(\mathcal{W}^{s,2}))}.
\end{align*}
Assume $f=f_1\otimes F_1$ and $g=g_1\otimes G_1$. Then, by relation $5.16$ of \cite{SKM93}:
\begin{align*}
<(I^{\alpha}_{+}\otimes E)^*(f);g>&_{(L^2(\mathbb{R})\otimes(\mathcal{W}^{s,2})^*,L^2(\mathbb{R})\otimes(\mathcal{W}^{s,2}))}=<f_1;I^{\alpha}_+(g_1)>_{L^2((a,b))}<F_1;G_1>_{((\mathcal{W}^{s,2})^*,(\mathcal{W}^{s,2}))},\\
&=<I^{\alpha}_{-}(\mathbb{I}_{(a,b)}f_1);g_1>_{L^2(\mathbb{R})}<F_1;G_1>_{((\mathcal{W}^{s,2})^*,(\mathcal{W}^{s,2}))},\\
&=<(I^{\alpha}_{-}(\mathbb{I}_{(a,b)}.)\otimes E)(f_1\otimes F_1);g_1\otimes G_1>_{(L^2(\mathbb{R})\otimes(\mathcal{W}^{s,2})^*,L^2(\mathbb{R})\otimes(\mathcal{W}^{s,2}))}.
\end{align*}
Thus, $(I^{\alpha}_{-}(\mathbb{I}_{(a,b)}.)\otimes E)$ and $(I^{\alpha}_{+}\otimes E)^*$ coincide on simple elements of $L^2((a,b))\otimes(\mathcal{W}^{s,2})^*$. Since, both operators are linear and bounded operators on $L^2((a,b))\otimes(\mathcal{W}^{s,2})^*$, they agree on $L^2((a,b))\otimes(\mathcal{W}^{s,2})^*$. Consequently, $(\nabla^{\alpha})^*$ is equal to $\nabla^*\circ (I^{\alpha}_{-}(\mathbb{I}_{(a,b)}.)\otimes E)$ on $L^2((a,b))\otimes (\mathcal{W}^{r-1,2})^*$. Moreover, it is well known that the operator $\nabla^*$ is a linear and continuous operator from $L^2(\mathbb{R})\otimes (\mathcal{W}^{r,2})$ into $(\mathcal{W}^{r-1,2})$ for any $r\geq 1$ (see Proposition $5.27$ of \cite{HKPS93}). This concludes the proof. 
\end{proof}

\begin{prop}\label{RepAdGrad1}
Let $\alpha\in (0,1/2)$ and $\Phi\in L^2((a,b))\otimes (\mathcal{W}^{1,2})$. We have:
\begin{align*}
(\nabla^{\alpha})^*(\Phi)=\int_a^bD^*_{\delta_t\circ I^{\alpha}_+}(\Phi_t)dt.
\end{align*}
\end{prop}

\begin{proof}
 Let $\xi\in S(\mathbb{R})$. $(\nabla^{\alpha})^*(\Phi)\in (L^2)\subset (\mathcal{W}^{1,2})^*$ and $:e^{<;\xi>}:$ is in $(S)\subset (\mathcal{W}^{s,2})$, for any $s\geq0$. Thus, we have:
\begin{align*}
S((\nabla^{\alpha}(\Phi))^*)(\xi)&=<(\nabla^{\alpha}(\Phi))^*,:e^{<;\xi>}:>_{(L^2)},\\
&=<(\nabla^{\alpha}(\Phi))^*,:e^{<;\xi>}:>_{((\mathcal{W}^{1,2})^*,(\mathcal{W}^{1,2}))},\\
&=<\Phi,\nabla^{\alpha}(:e^{<;\xi>}:)>_{L^2((a,b))\otimes(L^2)},\\
&=<\Phi, I^{\alpha}_{+}(\xi):e^{<;\xi>}:>_{L^2((a,b))\otimes(L^2)},\\
&=\int_a^bI^{\alpha}_{+}(\xi)(t)<\Phi_t, :e^{<;\xi>}:>_{(L^2)}dt,\\
&=\int_a^bS(\Phi_t)(\xi)I^{\alpha}_{+}(\xi)(t)dt,\\
&=S\bigg(\int_a^bD^*_{\delta_t\circ I^{\alpha}_{+}}(\Phi_t)dt\bigg)(\xi).
\end{align*}
\end{proof}


\section{From forward integrals to Wick-Itô integrals.}


\subsection{Fractional Brownian motion.}
\noindent
\begin{prop}\label{Decfbm1}
Let $\{\Phi_t: t\in(a,b)\}$ be a stochastic process such that for all $t\in (a,b)$, $\Phi_t\in (\mathcal{W}^{1,2})$. Then, we have, for every $\epsilon>0$ and $t\in(a,b)$:
\begin{align*}
\Phi_t\dfrac{B^H_{t+\epsilon}-B^H_t}{\epsilon}\overset{(S)^*}{=}\Phi_t\diamond\dfrac{B^{H}_{t+\epsilon}-B^H_t}{\epsilon}+A(H)\int_t^{t+\epsilon}\nabla^{H-\frac{1}{2}}(\Phi_t)(s,)\frac{ds}{\epsilon}.
\end{align*}
\end{prop}

\begin{proof}
Let $\epsilon>0$ be small enough such that $t+\epsilon\in(a,b)$. Since $\Phi_t\in(\mathcal{W}^{1,2})\subset(L^2)$, $\Phi_t=\sum_{n=0}^{\infty}I_n(\phi_n(,t))$ with $\phi_n(,t)\in \hat{L}^2(\mathbb{R}^n)$. Let us fix $n\geq 1$. Using the multiplication formula from Malliavin calculus (see Proposition $1.1.3$ of \cite{N06}), we obtain:
\begin{align*}
I_n(\phi_n(,t))(B^H_{t+\epsilon}-B^H_t)&=A(H)I_n(\phi_n(,t))I_1(\int_t^{t+\epsilon}\dfrac{(s-.)^{H-\frac{3}{2}}_{+}}{\Gamma(H-\frac{1}{2})}ds),\\
&=I_{n+1}(\phi_n(,t)\otimes g^H_{t,t+\epsilon})+nI_{n-1}(\phi_n(,t)\otimes_1 g^H_{t,t+\epsilon}),
\end{align*}
where $g^H_{t,t+\epsilon}(.)=A(H)/\Gamma(H-1/2)\int_t^{t+\epsilon}(s-.)^{H-3/2}_{+}ds$. Then, for any $\xi\in S(\mathbb{R})$, we have:
\begin{align*}
S(I_n(\phi_n(,t))(B^H_{t+\epsilon}-B^H_t))(\xi)=<\phi_n(,t);\xi^{\otimes n}><g^H_{t,t+\epsilon};\xi>+n<\phi_n(,t)\otimes_1 g^H_{t,t+\epsilon};\xi^{\otimes n-1}>.
\end{align*}
By continuity of the scalar product in $(L^2)$, we have:
\begin{align*}
S(\Phi_t(B^H_{t+\epsilon}-B^H_t))(\xi)=\sum_{n=0}^{\infty}\big(<\phi_n(,t);\xi^{\otimes n}><g^H_{t,t+\epsilon};\xi>+n<\phi_n(,t)\otimes_1 g^H_{t,t+\epsilon};\xi^{\otimes n-1}>\big).
\end{align*}
We want to compute separately the two infinite series appearing in the right hand side of the previous equality. First of all, we have:
\begin{align*}
\sum_{n=0}^{\infty}<\phi_n(,t);\xi^{\otimes n}><g^H_{t,t+\epsilon};\xi>&=S(\Phi_t)(\xi)S(B^H_{t+\epsilon}-B^H_t)(\xi),\\
&=S(\Phi_t\diamond (B^H_{t+\epsilon}-B^H_t))(\xi).
\end{align*}
Moreover, for the second term, we have, by Fubini theorem:
\begin{align*}
n<\phi_n(,t)\otimes_1 g^H_{t,t+\epsilon};\xi^{\otimes n-1}>&=A(H)n\int_{t}^{t+\epsilon}<I^{(0,...,0,H-\frac{1}{2})}_{+,...,+}(\phi_n(,t))(,s);\xi^{\otimes n-1}>ds,\\
&=A(H)\int_{t}^{t+\epsilon}S\big(nI_{n-1}(I^{(0,...,0,H-\frac{1}{2})}_{+,...,+}(\phi_n(,t))(,s))\big)(\xi)ds,\\
&=A(H)\int_{t}^{t+\epsilon}S(\nabla^{H-\frac{1}{2}}(I_n(\phi_n(,t)))(s,))(\xi)ds.
\end{align*}
We want to invert the infinite series and the integral over $(t,t+\epsilon)$. For any $n\geq 1$ and $t\in (a,b)$, we have:
\begin{align*}
|\int_{t}^{t+\epsilon}<I^{(0,...,0,H-\frac{1}{2})}_{+,...,+}(\phi_n(,t))(,s);\xi^{\otimes n-1}>ds|&\leq \int_{t}^{t+\epsilon}|<I^{(0,...,0,H-\frac{1}{2})}_{+,...,+}(\phi_n(,t))(,s);\xi^{\otimes n-1}>|ds,\\
&\leq \|\xi\|^{n-1}_{L^2(\mathbb{R})}\int_{t}^{t+\epsilon}\|I^{(0,...,0,H-\frac{1}{2})}_{+,...,+}(\phi_n(,t))(,s)\|_{L^2(\mathbb{R}^{n-1})}ds,\\
&\leq \|\xi\|^{n-1}_{L^2(\mathbb{R})}\int_{a}^{b}\|I^{(0,...,0,H-\frac{1}{2})}_{+,...,+}(\phi_n(,t))(,s)\|_{L^2(\mathbb{R}^{n-1})}ds<+\infty,
\end{align*}
since $I^{(0,...,0,H-1/2)}_{+,...,+}(\phi_n(,t))\in L^{(2,...,2,1/(1-H))}(\mathbb{R}^{n})$ and $1/(1-H)>2$ (see Theorem $24.1$ of \cite{SKM93}). Finally, we note that, using Cauchy-Schwarz and Jensen inequalities and Proposition \ref{ExGrad1}:
\begin{align*}
\sum_{n=1}^{+\infty}\bigg[n\|\xi\|^{n-1}_{L^2(\mathbb{R})}\int_{a}^{b}\|I^{(0,...,0,H-\frac{1}{2})}_{+,...,+}(\phi_n(,t))(,s)\|_{L^2(\mathbb{R}^{n-1})}ds\bigg]&\leq
\exp(\frac{1}{2}\|\xi\|^{2}_{L^2(\mathbb{R})})\Bigg(\sum_{n=1}^{\infty}n!n\Big(\int_a^b\\
&\|I_{+,...,+}^{(0,...,0,H-\frac{1}{2})}(\phi_n(,t))(.,s)\|_{L^2(\mathbb{R}^{n-1})}ds\Big)^2\Bigg)^{\frac{1}{2}},\\
\sum_{n=1}^{+\infty}\Bigg[n\bigg(\int_a^b\|I_{+,...,+}^{(0,...,0,H-\frac{1}{2})}(\phi_n(,t))(.,s)\|_{L^2(\mathbb{R}^{n-1})}ds\bigg)\|\xi\|^{n-1}_{L^2(\mathbb{R})}\Bigg]&\leq (b-a)^{\frac{1}{2}}\exp(\frac{1}{2}\|\xi\|^{2}_{L^2(\mathbb{R})})\times\\
&\Bigg(\sum_{n=1}^{\infty}n!n\int_a^b\|I_{+,...,+}^{(0,...,0,H-\frac{1}{2})}(\phi_n(,t))(.,s)\|^2_{L^2(\mathbb{R}^{n-1})}ds\Bigg)^{\frac{1}{2}},\\
\sum_{n=1}^{+\infty}\Bigg[n\bigg(\int_a^b\|I_{+,...,+}^{(0,...,0,H-\frac{1}{2})}(\phi_n(,t))(.,s)\|_{L^2(\mathbb{R}^{n-1})}ds\bigg)\|\xi\|^{n-1}_{L^2(\mathbb{R})}\Bigg]&\leq (b-a)^{\frac{1}{2}}\exp(\frac{1}{2}\|\xi\|^{2}_{L^2(\mathbb{R})})\|\nabla^{H-\frac{1}{2}}(\Phi_t)\|_{L^2((a,b))\otimes (L^2)},
\end{align*}
which is finite since $\Phi_t\in(\mathcal{W}^{1,2})$. Thus, we have:
\begin{align*}
\sum_{n=1}^{+\infty}n<\phi_n(,t)\otimes_1 g^H_{t,t+\epsilon};\xi^{\otimes n-1}>&=A(H)\int_{t}^{t+\epsilon}\sum_{n=1}^{+\infty}\big[S\big(nI_{n-1}(I^{(0,...,0,H-\frac{1}{2})}_{+,...,+}(\phi_n(,t))(,s))\big)(\xi)\big]ds,\\
&=A(H)\int_{t}^{t+\epsilon}S(\nabla^{H-\frac{1}{2}}(\Phi_t)(s,))(\xi)ds,\\
&=S(A(H)\int_{t}^{t+\epsilon}\nabla^{H-\frac{1}{2}}(\Phi_t)(s,)ds)(\xi),\\
\end{align*}
where we use the square-integrability of $\|\nabla^{H-\frac{1}{2}}(\Phi_t)(s,)\|_{(L^2)}$ on $(a,b)$ to justify the last equality. This concludes the proof.
\end{proof}

\begin{prop}\label{Decfbm2}
Let $\{\Phi_t: t\in(a,b)\}$ be a stochastic process such that for all $t\in (a,b)$, $\Phi_t\in (\mathcal{W}^{1,2})$. Moreover, assume that:
\begin{align*}
&\int_a^b\|\Phi_t\|^2_{(L^2)}dt<+\infty,\\
&\int_a^b\|\nabla^{H-\frac{1}{2}}(\Phi_t)\|^2_{L^2(a,b)\otimes (L^2)}dt<+\infty.
\end{align*}
Then, we have:
\begin{align*}
\int_a^b\Phi_t\dfrac{B^H_{t+\epsilon}-B^H_t}{\epsilon}dt\overset{(S)^*}{=}\int_a^b\Phi_t\diamond\dfrac{B^{H}_{t+\epsilon}-B^H_t}{\epsilon}dt+A(H)\int_a^b\int_t^{t+\epsilon}\nabla^{H-\frac{1}{2}}(\Phi_t)(s,)\frac{ds}{\epsilon}dt.
\end{align*}
\end{prop}

\begin{proof}
In order to prove the $(S)^*$-integrability of $\Phi_t \frac{B^{H}_{t+\epsilon}-B^H_t}{\epsilon}$ over $(a,b)$, we are going to prove the $(S)^*$-integrability of the two terms appearing in Proposition \ref{Decfbm1}. Let us start with the first term. For every $\xi\in S(\mathbb{R})$, we have:
\begin{align*}
S(\Phi_t\diamond\dfrac{B^{H}_{t+\epsilon}-B^H_t}{\epsilon})(\xi)&=S(\Phi_t)(\xi)S(\dfrac{B^{H}_{t+\epsilon}-B^H_t}{\epsilon})(\xi),\\
&=\mathbb{E}^{\mu}[\Phi_t\exp(<,\xi>-\frac{\|\xi\|_{L^2(\mathbb{R})}^2}{2})]A(H)\int_t^{t+\epsilon}I_+^{H-\frac{1}{2}}(\xi)(s)\dfrac{ds}{\epsilon}.
\end{align*}
By Cauchy-Schwarz inequality and $\|I_+^{H-1/2}(\xi)\|_{\infty}\leq C_H(\|\xi\|_{\infty}+\|\xi^{(1)}\|_{\infty}+\|\xi\|_{L^1(\mathbb{R})})$, we have:
\begin{align*}
|S(\Phi_t\diamond\dfrac{B^{H}_{t+\epsilon}-B^H_t}{\epsilon})(\xi)|&\leq C'_H(\|\xi\|_{\infty}+\|\xi^{(1)}\|_{\infty}+\|\xi\|_{L^1(\mathbb{R})})\exp(\frac{1}{2}\|\xi\|^2_{L^2(\mathbb{R})}) \mathbb{E}^{\mu}[\Phi_t^2]^{\frac{1}{2}}.\\
\end{align*}
Using the fact that $\int_a^b\|\Phi_t\|^2_{(L^2)}dt<+\infty$, we have the $(S)^*$-integrability of the first term. Let us deal with the second term now. For every $\xi\in S(\mathbb{R})$, we have:
\begin{align*}
|S(\int_{t}^{t+\epsilon}\nabla^{H-\frac{1}{2}}(\Phi_t)(s,)ds)(\xi)|&\leq \int_t^{t+\epsilon}|S(\nabla^{H-\frac{1}{2}}(\Phi_t)(s,))(\xi)|ds,\\
&\leq \exp(\frac{1}{2}\|\xi\|^2_{L^2(\mathbb{R})})\int_t^{t+\epsilon}\|\nabla^{H-\frac{1}{2}}(\Phi_t)(s,)\|_{(L^2)}ds,\\
&\leq \exp(\frac{1}{2}\|\xi\|^2_{L^2(\mathbb{R})})\int_a^{b}\|\nabla^{H-\frac{1}{2}}(\Phi_t)(s,)\|_{(L^2)}ds,\\
&\leq (b-a)^{\frac{1}{2}}\exp(\frac{1}{2}\|\xi\|^2_{L^2(\mathbb{R})})\|\nabla^{H-\frac{1}{2}}(\Phi_t)\|_{L^2((a,b))\otimes(L^2)}
\end{align*}
Using the fact that $\int_a^b\|\nabla^{H-\frac{1}{2}}(\Phi_t)\|^2_{L^2((a,b))\otimes(L^2)}dt<+\infty$, we have the $(S)^*$-integrability of the second term. This ends the proof. 
\end{proof}
\noindent
The next proposition examines the chain rule property for certain functionals of fractional Brownian motion.

\begin{prop}\label{Gradfbm}
Let $F$ be in $C^1(\mathbb{R})$. Assume that, for every $t\in (a,b)$, $F(B^H_t)\in(\mathcal{W}^{1,2})$. Then, we have, in $L^2((a,b))\otimes(L^2)$:
\begin{align*}
\nabla^{H-\frac{1}{2}}(F(B^H_t))(s,\omega)=\dfrac{A(H)\beta(H-\frac{1}{2},2-2H)}{(\Gamma(H-\frac{1}{2}))^2}\big(\int_0^t|s-r|^{2H-2}dr\big)F'(B^H_t)(\omega).
\end{align*}
\end{prop}

\begin{proof}
We denote by $\{H_{n,t^{2H}}:n\in\mathbb{N}\}$ the family of Hermite polynomials of parameter $t^{2H}$. Then, $\{H_{n,t^{2H}}/(\sqrt{n!}t^{Hn}):n\in\mathbb{N}\}$ is an orthonormal family of $L^2(\mathbb{R},\gamma_{t^{2H}}(dx))$ where $\gamma_{t^{2H}}$ is the centered Gaussian probability measure over $\mathbb{R}$ with variance $t^{2H}$. Since $\mathbb{E}[(F(B^H_t))^2]<\infty$, we have:
\begin{align*}
F=\sum_{n=0}^{\infty}c_{n,t^{2H}}\dfrac{H_{n,t^{2H}}}{\sqrt{n!}t^{Hn}}.
\end{align*}
where $c_{n,t^{2H}}$ are the Hermite coefficients of $F$ with respect to $\gamma_{t^{2H}}$. Using Wiener-Itô Theorem, we have:
\begin{align*}
F(B_t^H)=\sum_{n=0}^{\infty}c_{n,t^{2H}}\dfrac{I_n((g^H_{0,t})^{\otimes n})}{\sqrt{n!}t^{Hn}}.
\end{align*}
Let $n\geq 0$. Denote by $d_{n,t^{2H}}$ the Hermite coefficients of $F'$ with respect to $\gamma_{t^{2H}}$. We have:
\begin{align*}
d_{n,t^{2H}}=\int_{\mathbb{R}}F'(x)\dfrac{H_{n,t^{2H}}(x)}{\sqrt{n!}t^{Hn}}\gamma_{t^{2H}}(dx).
\end{align*}
Integrating by part, we obtain:
\begin{align*}
d_{n,t^{2H}}&=-\int_{\mathbb{R}}F(x)[nH_{n-1,t^{2H}}(x)-\dfrac{x}{t^{2H}}H_{n,t^{2H}}(x)]\dfrac{\gamma_{t^{2H}}(dx)}{\sqrt{n!}t^{Hn}},\\
d_{n,t^{2H}}&=\int_{\mathbb{R}}F(x)H_{n+1,t^{2H}}(x)\dfrac{\gamma_{t^{2H}}(dx)}{\sqrt{n!}t^{H(n+2)}},\\
d_{n,t^{2H}}&=\dfrac{\sqrt{n+1}}{t^H}c_{n+1,t^{2H}}.
\end{align*}
Since $\mathbb{E}[((N+E)^{1/2}F(B^H_t))^2]<+\infty$, $F'\in L^2(\mathbb{R},\gamma_{t^{2H}}(dx))$ and we have:
\begin{align*}
&F'=\sum_{n=0}^{\infty}\dfrac{\sqrt{n+1}}{t^H}c_{n+1,t^{2H}}\dfrac{H_{n,t^{2H}}}{\sqrt{n!}t^{Hn}},\\
&F'=\sum_{n=1}^{\infty}\dfrac{\sqrt{n}}{t^H}c_{n,t^{2H}}\dfrac{H_{n-1,t^{2H}}}{\sqrt{(n-1)!}t^{H(n-1)}}.
\end{align*}
Thus, by Wiener-Itô Theorem, we obtain:
\begin{align*}
&F'(B_t^H)=\sum_{n=1}^{\infty}\dfrac{\sqrt{n}}{t^H}c_{n,t^{2H}}\dfrac{I_{n-1}((g^H_{0,t})^{\otimes n-1})}{\sqrt{(n-1)!}t^{H(n-1)}},\\
&F'(B_t^H)=\sum_{n=1}^{\infty}n\dfrac{c_{n,t^{2H}}}{\sqrt{n!}t^{Hn}}I_{n-1}((g^H_{0,t})^{\otimes n-1}).
\end{align*}
Moreover, by assumption, $F(B_t^H)\in (\mathcal{W}^{1,2})$. Consequently, by Proposition \ref{ExGrad1}, we have, $\lambda\otimes\mu$-a.e:
\begin{align*}
&\nabla^{H-\frac{1}{2}}(F(B^H_t))(s,\omega)=\sum_{n=1}^{\infty}n\dfrac{c_{n,t^{2H}}}{\sqrt{n!}t^{Hn}}I_{n-1}(<\delta_s\circ I^{H-\frac{1}{2}}_+;(g^H_{0,t})^{\otimes n}>)(\omega),\\
&\nabla^{H-\frac{1}{2}}(F(B^H_t))(s,\omega)=I_+^{H-\frac{1}{2}}(g^H_{0,t})(s)\sum_{n=1}^{\infty}n\dfrac{c_{n,t^{2H}}}{\sqrt{n!}t^{Hn}}I_{n-1}((g^H_{0,t})^{\otimes n-1})(\omega),\\
&\nabla^{H-\frac{1}{2}}(F(B^H_t))(s,\omega)=\dfrac{A(H)\beta(H-\frac{1}{2},2-2H)}{(\Gamma(H-\frac{1}{2}))^2}\big(\int_0^t|s-r|^{2H-2}dr\big)F'(B^H_t)(\omega).
\end{align*}
\end{proof}

\begin{prop}\label{ConGradfbm}
Let $F$ be in $C^1(\mathbb{R})$. Assume that, for every $t\in(a,b)$, $F(B^H_t)\in(\mathcal{W}^{1,2})$. Then, we have:
\begin{align*}
A(H)\int_a^b\Big(\int_t^{t+\epsilon}\nabla^{H-\frac{1}{2}}(F(B^H_t))(s,)\frac{ds}{\epsilon}\Big)dt\overset{(L^2)}{\underset{\epsilon\rightarrow 0^+}{\longrightarrow}}H\int_a^bt^{2H-1}F'(B^H_t)dt,
\end{align*}
\end{prop}

\begin{proof} 
Using Proposition \ref{Gradfbm} as well as standard calculations lead to the convergence of this trace term.
\end{proof}

\begin{rem}\label{Cond}
Let $F$ be as in Proposition \ref{ConGradfbm}. In order to apply Proposition \ref{Decfbm2} to $\Phi_t=F(B^H_t)$, it is sufficient to prove that:
\begin{align*}
&\int_a^b\|F(B^H_t)\|^2_{(L^2)}dt<+\infty,\\
&\int_a^b\|\nabla^{H-\frac{1}{2}}(F(B^H_t))\|^2_{L^2((a,b);(L^2))}dt<+\infty.
\end{align*}
If we assume that there exist $C>0$ and $\lambda>0$ with $\lambda<1/(4b^{2H})$ such that:
\begin{align*}
\forall x\in\mathbb{R},\ \max\{F(x),F'(x)\}\leq Ce^{\lambda x^2}.
\end{align*}
Then, we have:
\begin{align*}
\int_a^b\|F(B^H_t)\|^2_{(L^2)}dt&\leq C\int_a^b\dfrac{dt}{\sqrt{1-4\lambda t^{2H}}}<\infty,\\
\int_a^b\|\nabla^{H-\frac{1}{2}}(F(B^H_t))\|^2_{L^2((a,b);(L^2))}dt&=\int_a^b\|I^{H-\frac{1}{2}}_+(g_t^H)\|^2_{L^2((a,b))}\|F'(B^H_t)\|^2_{(L^2)}dt,\\
&\leq C_{a,b,H}\int_a^b \dfrac{t^{2H}}{\sqrt{1-4\lambda t^{2H}}}dt<\infty.
\end{align*}
\end{rem}
\noindent
We conclude this subsection by the proof of Theorem \ref{Main1}.


\begin{proof}
We are going to prove that the term $\int_a^bF(B^H_t)\diamond (B^H_{t+\epsilon}-B^H_t)dt/\epsilon$ converges in $(L^2)$ to the following random variable:
\begin{align*}
(\nabla^{H-\frac{1}{2}})^*(F(B^H_.)).
\end{align*}
First of all, for every $t\in (a,b)$, since $F(B^H_t)\in (\mathcal{W}^{1,2})$:
\begin{align*}
F(B^H_t)\diamond \dfrac{B^H_{t+\epsilon}-B^H_t}{\epsilon}&=A(H)F(B^H_t)\diamond I_1(\int_t^{t+\epsilon}\frac{(s-.)^{H-\frac{3}{2}}_+}{\Gamma(H-\frac{1}{2})}\frac{ds}{\epsilon}),\\
&=\frac{A(H)}{\Gamma(H-\frac{1}{2})}\nabla^*\bigg(F(B^H_t)\int_t^{t+\epsilon}(s-.)^{H-\frac{3}{2}}_+\frac{ds}{\epsilon}\bigg).
\end{align*}
Indeed, using the $S$-transform, we have, for every $\xi\in S(\mathbb{R})$:
\begin{align*}
S(\nabla^*\bigg(F(B^H_t)\int_t^{t+\epsilon}(s-.)^{H-\frac{3}{2}}_+\frac{ds}{\epsilon}\bigg))(\xi)&=\mathbb{E}[\nabla^*\bigg(F(B^H_t)\int_t^{t+\epsilon}(s-.)^{H-\frac{3}{2}}_+\frac{ds}{\epsilon}\bigg):e^{\langle,\xi\rangle}:],\\
&=\mathbb{E}[F(B^H_t)\langle \int_t^{t+\epsilon}(s-.)^{H-\frac{3}{2}}_+\frac{ds}{\epsilon}, \xi(.)\rangle_{L^2(\mathbb{R})}:e^{\langle,\xi\rangle}:],\\
&=\mathbb{E}[F(B^H_t):e^{\langle,\xi\rangle}:]\Gamma(H-\frac{1}{2})\int_{t}^{t+\epsilon}I^{H-\frac{1}{2}}_+(\xi)(s)\frac{ds}{\epsilon},\\
&=\frac{\Gamma(H-\frac{1}{2})}{A(H)}S(F(B^H_t))(\xi)S(\dfrac{B^H_{t+\epsilon}-B^H_t}{\epsilon})(\xi).
\end{align*}
Therefore, we are left to prove that:
\begin{align*}
\frac{1}{\Gamma(H-\frac{1}{2})}\nabla^*\big(\int_a^bF(B^H_t)\bigg(\int_t^{t+\epsilon}(s-.)^{H-\frac{3}{2}}_+\frac{ds}{\epsilon}\bigg)dt\big)\underset{\epsilon\rightarrow 0^+}{\overset{(L^2)}{\longrightarrow}}\nabla^*(I_{-}^{H-\frac{1}{2}}(\mathbb{I}_{(a,b)}(.)F(B^H_.))).
\end{align*}
For this purpose, we are going to prove that:
\begin{align*}
\frac{1}{\Gamma(H-\frac{1}{2})}\int_a^bF(B^H_t)\bigg(\int_t^{t+\epsilon}(s-*)^{H-\frac{3}{2}}_+\frac{ds}{\epsilon}\bigg)dt\underset{\epsilon\rightarrow 0^+}{\longrightarrow}\int_a^b \dfrac{(t-*)^{H-\frac{3}{2}}_{+}}{\Gamma(H-\frac{1}{2})}F(B^H_t)dt,
\end{align*}
where the convergence holds in $L^2(\mathbb{R})\otimes(\mathcal{W}^{1,2})$. We denote by $(I)$ the square of the $L^2(\mathbb{R})\otimes(\mathcal{W}^{1,2})$-norm of the difference between the left-hand side and the right-hand side. We have:
\begin{align*}
(I)&=\int_{\mathbb{R}}\|\int_a^bF(B^H_t)\bigg(\int_t^{t+\epsilon}((s-x)^{H-\frac{3}{2}}_+-(t-x)^{H-\frac{3}{2}}_{+})\frac{ds}{\epsilon}\bigg)dt\|^2_{(\mathcal{W}^{1,2})}dx,\\
(I)&=\int_{\mathbb{R}}\langle\int_a^bF(B^H_{t_1})\bigg(\int_{t_1}^{t_1+\epsilon}((s-x)^{H-\frac{3}{2}}_+-(t_1-x)^{H-\frac{3}{2}}_{+})\frac{ds}{\epsilon}\bigg)dt_1;\\
&\int_a^bF(B^H_{t_2})\bigg(\int_{t_2}^{t_2+\epsilon}((r-x)^{H-\frac{3}{2}}_+-(t_2-x)^{H-\frac{3}{2}}_{+})\frac{dr}{\epsilon}\bigg)dt_2\rangle_{(\mathcal{W}^{1,2})}dx,\\
(I)&=\int_{(a,b)\times(a,b)}\langle F(B^H_{t_1}),F(B^H_{t_2})\rangle_{(\mathcal{W}^{1,2})}\bigg(\int_{\mathbb{R}}\bigg(\int_{(t_1,t_1+\epsilon)\times(t_2,t_2+\epsilon)}\big((s-x)^{H-\frac{3}{2}}_+-(t_1-x)^{H-\frac{3}{2}}_{+}\big)\\
&\times\big((r-x)^{H-\frac{3}{2}}_+-(t_2-x)^{H-\frac{3}{2}}_{+}\big)\frac{dsdr}{\epsilon^2}\bigg)dx\bigg) dt_1dt_2,\\
(I)&=\beta(H-\frac{1}{2},2-2H)\int_{(a,b)\times(a,b)}\langle F(B^H_{t_1}),F(B^H_{t_2})\rangle_{(\mathcal{W}^{1,2})}\bigg(\int_{t_1}^{t_1+\epsilon}\int_{t_2}^{t_2+\epsilon}\big[|s-r|^{2H-2}-|t_2-s|^{2H-2}\\
&-|t_1-r|^{2H-2}+|t_1-t_2|^{2H-2}\big]\frac{dsdr}{\epsilon^2}\bigg)dt_1dt_2.
\end{align*}
For the last equality, we have used the following relation which holds for every $\gamma\in (-1,-1/2)$:
\begin{align*}
\int_{\mathbb{R}}(s_1-u)^{\gamma}_+(s_2-u)^{\gamma}_+du= \beta(\gamma+1,-2\gamma-1)|s_2-s_1|^{2\gamma+1}.
\end{align*}
Moreover, by Cauchy-Schwarz inequality, we have:
\begin{align*}
|\langle F(B^H_{t_1}),F(B^H_{t_2})\rangle_{(\mathcal{W}^{1,2})}|\leq \|F(B^H_{t_1})\|_{(\mathcal{W}^{1,2})}\|F(B^H_{t_2})\|_{(\mathcal{W}^{1,2})},
\end{align*}
And, by Meyer inequality, we have, for every $i\in \{1,2\}$:
\begin{align*}
\|F(B^H_{t_i})\|_{(\mathcal{W}^{1,2})}&\leq C\bigg(\|F(B^H_{t_i})\|^2_{(L^2)}+\|\nabla(F(B^H_{t_i}))\|^2_{L^2(\mathbb{R})\otimes(L^2)}\bigg)^{\frac{1}{2}},\\
&\leq C\bigg(\dfrac{1}{\sqrt{1-4\lambda t_i^{2H}}}+\dfrac{t_i^{2H}}{\sqrt{1-4\lambda t_i^{2H}}}\bigg)^{\frac{1}{2}},\\
&\leq C(1+t_i^{2H})^{\frac{1}{2}}\bigg(\dfrac{1}{1-4\lambda t_i^{2H}}\bigg)^{\frac{1}{4}}.
\end{align*}
Thus:
\begin{align*}
(I)&\leq C\int_{(a,b)\times(a,b)}\prod_{i=1}^2\bigg((1+t_i^{2H})^{\frac{1}{2}}\bigg(\dfrac{1}{1-4\lambda t_i^{2H}}\bigg)^{\frac{1}{4}}\bigg)\bigg|\bigg(\int_{t_1}^{t_1+\epsilon}\int_{t_2}^{t_2+\epsilon}\big[|s-r|^{2H-2}-|t_2-s|^{2H-2}\\
&-|t_1-r|^{2H-2}+|t_1-t_2|^{2H-2}\big]\frac{dsdr}{\epsilon^2}\bigg)\bigg|dt_1dt_2.
\end{align*}
For every $t_1,t_2$ in $(a,b)\times (a,b)$ such that $t_1\ne t_2$, we have:
\begin{align*}
\bigg(\int_{t_1}^{t_1+\epsilon}\int_{t_2}^{t_2+\epsilon}\big[|s-r|^{2H-2}-|t_2-s|^{2H-2}-|t_1-r|^{2H-2}+|t_1-t_2|^{2H-2}\big]\frac{dsdr}{\epsilon^2}\bigg)\underset{\epsilon\rightarrow 0^+}{\longrightarrow}0.
\end{align*}
To pursue, we need to provide an upper bound for:
\begin{align*}
(II)=\int_{t_1}^{t_1+\epsilon}\int_{t_2}^{t_2+\epsilon}\big[|s-r|^{2H-2}-|t_2-s|^{2H-2}-|t_1-r|^{2H-2}+|t_1-t_2|^{2H-2}\big]\frac{dsdr}{\epsilon^2}.
\end{align*} 
We assume without loss of generality that $t_1\ne t_2$, $t_1> t_2$ and $\epsilon< (t_1-t_2)/2$. Note that $t_2+\epsilon< t_1$. We have, for every $(s,r)\in (t_1,t_1+\epsilon)\times(t_2,t_2+\epsilon)$, since $2H-2<0$:
\begin{align*}
&(t_1+\epsilon-t_2)^{2H-2}\leq(s-r)^{2H-2}\leq (t_1-t_2-\epsilon)^{2H-2},\\
&(t_1+\epsilon-t_2)^{2H-2}\leq(s-t_2)^{2H-2}\leq (t_1-t_2)^{2H-2},\\
&(t_1-t_2)^{2H-2}\leq(t_1-r)^{2H-2}\leq (t_1-t_2-\epsilon)^{2H-2}.
\end{align*}
Thus,
\begin{align*}
(II)&\leq (t_1-t_2-\epsilon)^{2H-2}-(t_1+\epsilon-t_2)^{2H-2},\\
&\leq (t_1-t_2)^{2H-2}[(1-\frac{\epsilon}{t_1-t_2})^{2H-2}-(1+\frac{\epsilon}{t_1-t_2})^{2H-2}],\\
&\leq (t_1-t_2)^{2H-2}\underset{\epsilon\in [0,\frac{1}{2}]}{\max}[(1-\epsilon)^{2H-2}-(1+\epsilon)^{2H-2}].
\end{align*}
Consequently, we have the following upper bound:
\begin{align*}
\bigg|\bigg(\int_{t_1}^{t_1+\epsilon}\int_{t_2}^{t_2+\epsilon}\big[|s-r|^{2H-2}-|t_2-s|^{2H-2}-|t_1-r|^{2H-2}+|t_1-t_2|^{2H-2}\big]\frac{dsdr}{\epsilon^2}\bigg)\bigg|\leq C' |t_1-t_2|^{2H-2}.
\end{align*}
To conclude, we need to prove the finiteness of the following integral:
\begin{align*}
\int_{(a,b)\times(a,b)}\prod_{i=1}^2\bigg((1+t_i^{2H})^{\frac{1}{2}}\bigg(\dfrac{1}{1-4\lambda t_i^{2H}}\bigg)^{\frac{1}{4}}\bigg)|t_1-t_2|^{2H-2}dt_1dt_2.
\end{align*}
Denoting by $h(t)=(1+t^{2H})^{1/2}\big(1/(1-4\lambda t^{2H})\big)^{1/4}$ and using Hardy-Littlewood-Sobolev inequality (Theorem 4.3 page 106 of \cite{LL01}), we obtain:
\begin{align*}
\int_{(a,b)\times(a,b)}h(t_1)h(t_2)|t_1-t_2|^{2H-2}dt_1dt_2\leq C_H \|h\|^2_{L^{1/H}((a,b))}
\end{align*}
But $1/H\in (1,2)$, we have:
\begin{align*}
\int_{(a,b)\times(a,b)}\prod_{i=1}^2\bigg((1+t_i^{2H})^{\frac{1}{2}}\bigg(\dfrac{1}{1-4\lambda t_i^{2H}}\bigg)^{\frac{1}{4}}\bigg)|t_1-t_2|^{2H-2}dt_1dt_2\leq C_{H,a,b}\|h\|^2_{L^{2}((a,b))}<+\infty.
\end{align*}
By Lebesgue dominated convergence theorem, we obtain the desired convergence. Applying Proposition \ref{ConGradfbm} and Remark \ref{Cond} leads to the result.
\end{proof}


\subsection{The Rosenblatt process.}
Before stating the first result of this section, we introduce a useful notation. For any $t\in (a,b)$ and for any $\epsilon>0$, we define:
\begin{align*}
h^H_{t,t+\epsilon}(x_1,x_2)=d(H)\int_t^{t+\epsilon}\dfrac{(s-x_1)^{\frac{H}{2}-1}_+}{\Gamma\big(\frac{H}{2}\big)}\dfrac{(s-x_2)^{\frac{H}{2}-1}_+}{\Gamma\big(\frac{H}{2}\big)}ds
\end{align*} 

\begin{prop}\label{DecRos1}
Let $t\in (a,b)$. Let $\epsilon>0$ such that $t+\epsilon<b$. Let $F$ be in $C^{\infty}\big(\mathbb{R}\big)$ with polynomial growth (as well as its derivatives). Then, we have:
\begin{align*}
F(X^H_t)\dfrac{X^H_{t+\epsilon}-X^H_t}{\epsilon}&\overset{(S)^*}{=}F(X^H_t)\diamond\dfrac{X^H_{t+\epsilon}-X^H_t}{\epsilon}+2d(H)\int_t^{t+\epsilon}D^*_{\delta_s\circ I^{\frac{H}{2}}_+}(\nabla^{\frac{H}{2}}(F(X^H_t))(s,.))\frac{ds}{\epsilon}\\
&+\langle \nabla^{(2)}\big(F(X^H_t)\big); h^H_{t,t+\epsilon}\rangle_{\hat{L}^2(\mathbb{R}^2)}.
\end{align*}
\end{prop}

\begin{proof}
Let $\xi\in S\big(\mathbb{R}\big)$. We have:
\begin{align*}
S\big(F(X^H_t)(X^H_{t+\epsilon}-X^H_t)\big)(\xi)&=\mathbb{E}\big[:e^{\langle ;\xi\rangle}: F(X^H_t) I_2\big(h^H_{t,t+\epsilon}\big)\big],\\
&=\mathbb{E}\big[:e^{\langle ;\xi\rangle}: F(X^H_t) \big(\nabla^{(2)}\big)^*\big(h^H_{t,t+\epsilon}\big)\big],\\
&=\langle :e^{\langle ;\xi\rangle}: F(X^H_t); \big(\nabla^{(2)}\big)^*\big(h^H_{t,t+\epsilon}\big)\rangle_{(L^2)},\\
&=\langle \nabla^{(2)}\big(:e^{\langle ;\xi\rangle}: F(X^H_t)\big); h^H_{t,t+\epsilon}\rangle_{\hat{L}^2(\mathbb{R}^2)\otimes(L^2)},\\
&= \langle \xi^{\otimes 2}:e^{\langle ;\xi \rangle}: F(X^H_t)+2 \xi :e^{\langle ;\xi \rangle}: \nabla\big(F(X^H_t)\big);h^H_{t,t+\epsilon}\rangle_{\hat{L}^2(\mathbb{R}^2)\otimes(L^2)}\\
&+\langle:e^{\langle ;\xi \rangle}: \nabla^{(2)}\big(F(X^H_t)\big);h^H_{t,t+\epsilon}\rangle_{\hat{L}^2(\mathbb{R}^2)\otimes(L^2)}.
\end{align*}
Now we compute separately the three terms which appear in the previous sum. For the first term, we have:
\begin{align*}
\langle \xi^{\otimes 2}:e^{\langle ;\xi \rangle}: F(X^H_t);h^H_{t,t+\epsilon}\rangle_{\hat{L}^2(\mathbb{R}^2)\otimes(L^2)}&=S\big(F(X^H_t)\big)(\xi)S\big(I_2(h^H_{t,t+\epsilon})\big)(\xi),\\
&=S\big(F(X^H_t)\diamond(X^H_{t+\epsilon}-X^H_t)\big)(\xi).
\end{align*}
For the second term, we have:
\begin{align*}
\langle \xi :e^{\langle ;\xi \rangle}: \nabla\big(F(X^H_t)\big);h^H_{t,t+\epsilon}\rangle_{\hat{L}^2(\mathbb{R}^2)\otimes (L^2)}&=\langle \xi \otimes \mathbb{E}\big[:e^{\langle ;\xi\rangle}: \nabla\big(F(X^H_t)\big)\big] ; h^H_{t,t+\epsilon} \rangle_{\hat{L}^2(\mathbb{R}^2)},\\
&=\int_{\mathbb{R}^2}\xi(x_1)\mathbb{E}\big[:e^{\langle ;\xi\rangle}: \nabla\big(F(X^H_t)(x_2,.)\big)\big] h^H_{t,t+\epsilon}(x_1,x_2)dx_1\otimes dx_2,\\
&=d(H)\int_{\mathbb{R}^2}\xi(x_1)G(t,x_2) \big(\int_t^{t+\epsilon}\dfrac{(s-x_1)^{\frac{H}{2}-1}_+}{\Gamma(H/2)}\dfrac{(s-x_2)^{\frac{H}{2}-1}_+}{\Gamma(H/2)}ds\big)dx_1\otimes dx_2,\\
&=d(H)\int_t^{t+\epsilon}I_+^{\frac{H}{2}}(\xi)(s)I^{\frac{H}{2}}_+(G(t,.))(s)ds,\\
&=d(H)\int_t^{t+\epsilon}I_+^{\frac{H}{2}}(\xi)(s)S\big(\nabla^{\frac{H}{2}}\big(F(X^H_t)\big)(s,)\big)(\xi)ds,\\
&=d(H)S\bigg(\int_t^{t+\epsilon}D^*_{\delta_s\circ I^{\frac{H}{2}}_+}(\nabla^{\frac{H}{2}}(F(X^H_t))(s,.))ds\bigg)(\xi),
\end{align*}
where we have set,
\begin{align*}
G(t,x_2)=\mathbb{E}\big[:e^{\langle ;\xi\rangle}: \nabla\big(F(X^H_t)(x_2,.)\big)\big],
\end{align*}
and used successively Fubini Theorem twice and the $(S)^*$-integrability of,
\begin{align*}
s\rightarrow D^*_{\delta_s\circ I^{\frac{H}{2}}_+}(\nabla^{\frac{H}{2}}(F(X^H_t))(s,.)),
\end{align*} 
over $(t,t+\epsilon)$. We note in particular that the key fact is that $\big(I^{H/2}_+\otimes E\big) \circ \mid \nabla\big(F(X^H_t)\big)\mid$ is in $L^2\big((a,b)\big)\otimes(L^2)$. For the third term, we have:
\begin{align*}
\langle:e^{\langle ;\xi \rangle}: \nabla^{(2)}\big(F(X^H_t)\big);h^H_{t,t+\epsilon}\rangle_{\hat{L}^2(\mathbb{R}^2)\otimes(L^2)}&=S\big(\langle \nabla^{(2)}\big(F(X^H_t)\big) ;h^H_{t,t+\epsilon}\rangle_{\hat{L}^2(\mathbb{R}^2)}\big)(\xi).
\end{align*}
\end{proof}
 \noindent
In the subsequent subsections, we analyse independently the three terms in the previous decomposition. In particular, we will prove the following strong convergence result which allows us to obtain the explicit decomposition of the forward integral of $F(X^H_t)$ with respect to $X^H_t$.

\begin{prop}\label{StrConv}
We have, in $(L^2)$:
\begin{align*}
&\int_a^b F(X^H_t)\diamond\dfrac{X^H_{t+\epsilon}-X^H_t}{\epsilon}dt \underset{\epsilon\rightarrow 0^+}{\longrightarrow} d(H)\big(\nabla^{(2)}\big)^*\bigg(\int_a^b F(X^H_t)\dfrac{(t-.)^{\frac{H}{2}-1}_+(t-\#)^{\frac{H}{2}-1}_+}{\Gamma(H/2)^2}dt\bigg),\\
&2d(H)\int_a^b2\bigg(\int_t^{t+\epsilon}D^*_{\delta_s\circ I^{\frac{H}{2}}_+}(\nabla^{\frac{H}{2}}(F(X^H_t))(s,.))\frac{ds}{\epsilon}\bigg)dt \underset{\epsilon\rightarrow 0^+}{\longrightarrow} B(H)\nabla^*\bigg(\int_a^b F'(X^H_t)(t-.)^{\frac{H}{2}-1}_+I_1(l^H_{t,t})dt\bigg),\\
&\langle \nabla^{(2)}\big(F(X^H_t)\big); h^H_{t,t+\epsilon}\rangle_{\hat{L}^2(\mathbb{R}^2)}\underset{\epsilon\rightarrow 0^+}{\longrightarrow} H\int_{a}^bt^{2H-1}F'(X^H_t)dt+\frac{H}{2}\kappa_3(X^H_1)\int_a^bt^{3H-1}F^{(2)}(X^H_t)dt\\
&+C(H)\int_a^bI_2(e^H_{t,t})F^{(2)}(X^H_t)dt,
\end{align*}
with,
\begin{align*}
l^H_{t,t}(x)&=\int_0^t(u-x)^{H/2-1}_+|t-u|^{H-1}du,\\
e^H_{t,t}(x_1,x_2)&=\int_0^t\int_0^t(u-x_1)^{\frac{H}{2}-1}_+(v-x_2)^{\frac{H}{2}-1}_+\mid u-t\mid^{H-1} \mid v-t\mid^{H-1}dudv,\\
B(H)&=\dfrac{4d(H)}{\big(\Gamma(\frac{H}{2})\big)^2}\sqrt{\dfrac{H(2H-1)}{2}},\\
C(H)&=\dfrac{2 d(H)}{\big(\Gamma(\frac{H}{2})\big)^2}H(2H-1).
\end{align*}
\end{prop}

\begin{proof}
This a combination of Propositions \ref{DivConvRos} and \ref{TraceConvRos} and Lemmas \ref{Trace2ConvRos1} and \ref{Trace2ConvRos2}.
\end{proof}
\noindent
As a direct application of Propositions \ref{DecRos1} and \ref{StrConv}, we obtain the proof of Theorem \ref{Main2}.


\subsubsection{Second order divergence term}
In this section, we prove the following strong convergence result:
\begin{align*}
\int_a^b F(X^H_t)\diamond\dfrac{X^H_{t+\epsilon}-X^H_t}{\epsilon}dt \underset{\epsilon\rightarrow 0^+}{\overset{(L^2)}{\longrightarrow}} d(H)\big(\nabla^{(2)}\big)^*\bigg(\int_a^b F(X^H_t)\dfrac{(t-.)^{\frac{H}{2}-1}_+(t-\#)^{\frac{H}{2}-1}_+}{\Gamma(H/2)^2}dt\bigg).
\end{align*}
For this purpose, we proceed in the following way:
\begin{itemize}
\item We find a new representation for $\int_a^b F(X^H_t)\diamond\dfrac{X^H_{t+\epsilon}-X^H_t}{\epsilon}dt$ ensuring that it is an element in $(L^2)$
\item Then, we prove the wanted convergence result
\end{itemize}
First, we have the following lemma.

\begin{lem}\label{RepDivRos}
Let $\epsilon>0$. We have:
\begin{align*}
\int_a^b F(X^H_t)\diamond\dfrac{X^H_{t+\epsilon}-X^H_t}{\epsilon}dt=\big(\nabla^{(2)}\big)^*\bigg(\int_a^bF(X^H_t)\dfrac{h^H_{t,t+\epsilon}}{\epsilon}dt\bigg)
\end{align*}
\end{lem}

\begin{proof}
Let $\xi\in S(\mathbb{R})$. Then, we have:
\begin{align*}
S\bigg(\big(\nabla^{(2)}\big)^*\bigg(\int_a^b F(X^H_t)\dfrac{h^H_{t,t+\epsilon}}{\epsilon}dt\bigg)\bigg)(\xi)&=\langle \big(\nabla^{(2)}\big)^*\bigg(\int_a^b F(X^H_t)\dfrac{h^H_{t,t+\epsilon}}{\epsilon}dt\bigg); :e^{\langle ;\xi \rangle}:\rangle_{(L^2)},\\
&=\langle \int_a^b F(X^H_t)\dfrac{h^H_{t,t+\epsilon}}{\epsilon}dt ; \nabla^{(2)}\big(:e^{\langle ;\xi \rangle}:\big)\rangle_{\hat{L}^2\big(\mathbb{R}^2\big)\otimes(L^2)},\\
&=\langle \int_a^b F(X^H_t)\dfrac{h^H_{t,t+\epsilon}}{\epsilon}dt ; \xi^{\otimes 2} :e^{\langle ;\xi\rangle}:\rangle_{\hat{L}^2\big(\mathbb{R}^2\big)\otimes(L^2)},\\
&=\int_a^bS(F(X^H_t))(\xi)S(X^H_{t+\epsilon}-X^H_t)(\xi)\frac{dt}{\epsilon},\\
&=S\bigg(\int_a^bF(X^H_t)\diamond\dfrac{X^H_{t+\epsilon}-X^H_t}{\epsilon}dt\bigg)(\xi)
\end{align*}
\end{proof}

\begin{prop}\label{DivConvRos}
We have in $(L^2)$:
\begin{align*}
\big(\nabla^{(2)}\big)^*\bigg(\int_a^bF(X^H_t)\dfrac{h^H_{t,t+\epsilon}}{\epsilon}dt\bigg)\underset{\epsilon\rightarrow 0^+}{\overset{(L^2)}{\longrightarrow}} d(H)\big(\nabla^{(2)}\big)^*\bigg(\int_a^b F(X^H_t)\dfrac{(t-.)^{\frac{H}{2}-1}_+(t-\#)^{\frac{H}{2}-1}_+}{\Gamma(H/2)^2}dt\bigg).
\end{align*}
\end{prop}

\begin{proof}
By the second part of Proposition \ref{ConAdOp}, we have to prove the following convergence in $\hat{L}^2\big(\mathbb{R}^2\big)\otimes \big(\mathcal{W}^{2,2}\big)$:
\begin{align*}
\int_{a}^bF(X^H_t)\dfrac{h^H_{t,t+\epsilon}}{\epsilon}dt\underset{\epsilon\rightarrow 0^+}{\longrightarrow}\int_a^bF(X^H_t)\dfrac{(t-.)^{\frac{H}{2}-1}_+(t-\#)^{\frac{H}{2}-1}_+}{\Gamma(\frac{H}{2})^2}dt.
\end{align*}
For this purpose, we proceed as in the proof of Theorem \ref{Main1}. Let us consider the following quantity:
\begin{align*}
J_{\epsilon}=\int_{\mathbb{R}^2}\mid\mid \int_a^bF(X^H_t)\bigg(\dfrac{h^H_{t,t+\epsilon}}{\epsilon}-\dfrac{(t-x_1)^{\frac{H}{2}-1}_+(t-x_2)^{\frac{H}{2}-1}_+}{\Gamma(\frac{H}{2})^2}\bigg)dt\mid\mid^2_{(\mathcal{W}^{2,2})}dx_1dx_2.
\end{align*}
Then, we have:
\begin{align*}
J_{\epsilon}&=\int_{\mathbb{R}^2}\langle \int_a^bF(X^H_{t_1})\bigg(\int_{t_1}^{t_1+\epsilon}\dfrac{(s-x_1)^{\frac{H}{2}-1}_+(s-x_2)^{\frac{H}{2}-1}_+}{\Gamma(\frac{H}{2})^2}-\dfrac{(t_1-x_1)^{\frac{H}{2}-1}_+(t_1-x_2)^{\frac{H}{2}-1}_+}{\Gamma(\frac{H}{2})^2}\frac{ds}{\epsilon}\bigg)dt_1;\int_a^bF(X^H_{t_2})\\
&\times\bigg(\int_{t_2}^{t_2+\epsilon}\dfrac{(s-x_1)^{\frac{H}{2}-1}_+(s-x_2)^{\frac{H}{2}-1}_+}{\Gamma(\frac{H}{2})^2}-\dfrac{(t_2-x_1)^{\frac{H}{2}-1}_+(t_2-x_2)^{\frac{H}{2}-1}_+}{\Gamma(\frac{H}{2})^2}\frac{ds}{\epsilon}\bigg)dt_2\rangle_{(\mathcal{W}^{2,2})}dx_1dx_2,\\
&=\int_{(a,b)\times (a,b)}\langle F(X^H_{t_1}); F(X^H_{t_2})\rangle_{(\mathcal{W}^{2,2})}\bigg(\int_{t_1}^{t_1+\epsilon}\int_{t_2}^{t_2+\epsilon}\int_{\mathbb{R}^2}\bigg(\dfrac{(s-x_1)^{\frac{H}{2}-1}_+(s-x_2)^{\frac{H}{2}-1}_+}{\Gamma(\frac{H}{2})^2}\\
&-\dfrac{(t_1-x_1)^{\frac{H}{2}-1}_+(t_1-x_2)^{\frac{H}{2}-1}_+}{\Gamma(\frac{H}{2})^2}\bigg)\bigg(\dfrac{(r-x_1)^{\frac{H}{2}-1}_+(r-x_2)^{\frac{H}{2}-1}_+}{\Gamma(\frac{H}{2})^2}-\dfrac{(t_2-x_1)^{\frac{H}{2}-1}_+(t_2-x_2)^{\frac{H}{2}-1}_+}{\Gamma(\frac{H}{2})^2}\bigg)dx_1dx_2\frac{dsdr}{\epsilon^2}\bigg)dt_1dt_2
\end{align*}
Straightforward computations as in the proof of Theorem \ref{Main1} lead to the following equality:
\begin{align*}
J_{\epsilon}=\frac{\beta\big(\frac{H}{2};1-H\big)^2}{\Gamma\big(\frac{H}{2}\big)^4}\int_{(a,b)\times (a,b)}\langle F(X^H_{t_1}); F(X^H_{t_2})\rangle_{(\mathcal{W}^{2,2})}&\bigg(\int_{t_1}^{t_1+\epsilon}\int_{t_2}^{t_2+\epsilon}\bigg[\mid s-r\mid^{2H-2}-\mid t_1-r\mid^{2H-2}\\
&-\mid t_2-s\mid^{2H-2}+\mid t_2-t_1\mid^{2H-2}\bigg]\frac{dsdr}{\epsilon^2}\bigg)dt_1dt_2.
\end{align*}
To conclude, we proceed as in the proof of Theorem \ref{Main1}. First of all, by Cauchy-Schwarz inequality, we have:
\begin{align*}
\mid\langle F(X^H_{t_1}); F(X^H_{t_2})\rangle_{(\mathcal{W}^{2,2})}\mid\leq \mid\mid F(X^H_{t_1})\mid\mid_{(\mathcal{W}^{2,2})} \mid\mid F(X^H_{t_2})\mid\mid_{(\mathcal{W}^{2,2})}
\end{align*}
Now, by Meyer inequality, we have:
\begin{align*}
\mid\mid F(X^H_{t_i})\mid\mid_{(\mathcal{W}^{2,2})}^2\leq C \bigg(\mid\mid F(X^H_{t_i})\mid\mid^2_{(L^2)}+\mid\mid \nabla\big(F(X^H_{t_i})\big)\mid\mid^2_{L^2(\mathbb{R})\otimes(L^2)}+\mid\mid \nabla^{(2)}\big(F(X^H_{t_i})\big)\mid\mid^2_{\hat{L}^2(\mathbb{R}^2)\otimes(L^2)}\bigg).
\end{align*}
Since $F$ is infinitely differentiable with polynomial growth (and its derivatives as well) we have the following bounds:
\begin{align*}
\mid\mid F(X^H_{t})\mid\mid^2_{(L^2)}&\leq C\mathbb{E}\big[\big(1+\mid X^H_{t}\mid\big)^{2N_1}\big],\\
&\leq C_1 P_{2N_1}(t^H),\\
\mid\mid \nabla\big(F(X^H_{t})\big)\mid\mid^2_{L^2(\mathbb{R})\otimes(L^2)}&=4\mathbb{E}\bigg[F'(X^H_{t})^2I_2\bigg(\int_{\mathbb{R}}h^H_t(x,.)^{\otimes 2}dx\bigg)\bigg]+4\mid\mid h^H_t\mid\mid_{L^2(\mathbb{R}^2)}^2\mathbb{E}[F'(X^H_{t})^2],\\
&\leq 4\mathbb{E}\bigg[F'(X^H_{t})^2I_2\bigg(\int_{\mathbb{R}}h^H_t(x,.)^{\otimes 2}dx\bigg)\bigg]+C_2t^{2H}Q_{2N_2}(t^H),\\
&\leq 4\mathbb{E}\big[F'(X^H_{t})^4\big]^{\frac{1}{2}}\mathbb{E}\bigg[I_2\bigg(\int_{\mathbb{R}}h^H_t(x,.)^{\otimes 2}dx\bigg)^2\bigg]^{\frac{1}{2}}+C_2t^{2H}Q_{2N_2}(t^H),\\
&\leq C_3 \big(R_{4N_2}(t^H)\big)^{\frac{1}{2}} \bigg(\int_{\mathbb{R}^2}\bigg(\int_{\mathbb{R}}h^H_t(x,x_1)h^H_t(x,x_2)dx\bigg)^2dx_1dx_2\bigg)^{\frac{1}{2}}+C_2t^{2H}Q_{2N_2}(t^H),\\
&\leq C_3(H) \big(R_{4N_2}(t^H)\big)^{\frac{1}{2}}\bigg(\int_0^t\int_0^t\int_0^t\int_0^t\mid s_1-r_1 \mid^{H-1} \mid s_2-r_2\mid^{H-1}\\
&\times\mid s_1-s_2\mid^{H-1}\mid r_2-r_1\mid^{H-1}ds_1ds_2dr_1dr_2\bigg)^{\frac{1}{2}}+C_2t^{2H}Q_{2N_2}(t^H),\\
&\leq C_3(H) \big(R_{4N_2}(t^H)\big)^{\frac{1}{2}} t^{2H}+C_2t^{2H}Q_{2N_2}(t^H),
\end{align*}
with $N_1,N_2\geq 1$ and $P_{2N_1}$, $Q_{2N_2}$ and $R_{4N_2}$ are polynomials of respective degrees $2N_1$ $2N_2$ and $4N_2$. We obtain a similar bound for $\mid\mid \nabla^{(2)}\big(F(X^H_{t_i})\big)\mid\mid^2_{\hat{L}^2(\mathbb{R}^2)\otimes(L^2)}$. Thus, we have the following bound:
\begin{align*}
\mid\langle F(X^H_{t_1}); F(X^H_{t_2})\rangle_{(\mathcal{W}^{2,2})}\mid\leq C_H G(t_1^H)G(t_2^H),
\end{align*}
where $G$ is a positive valued function with power growth at most. This implies, the following estimate on $\mid J_\epsilon\mid$:
\begin{align*}
\mid J_\epsilon\mid \leq C_H \int_{(a,b)\times (a,b)} &G(t_1^H)G(t_2^H)\bigg| \int_{t_1}^{t_1+\epsilon}\int_{t_2}^{t_2+\epsilon}\bigg[\mid s-r\mid^{2H-2}-\mid t_1-r\mid^{2H-2}\\
&-\mid t_2-s\mid^{2H-2}+\mid t_2-t_1\mid^{2H-2}\bigg]\frac{dsdr}{\epsilon^2}\bigg| dt_1dt_2.
\end{align*}
Thanks to Lebesgue dominated convergence Theorem and similar arguments as in the proof of Theorem \ref{Main1}, we obtain the result.
\end{proof}


\subsubsection{Trace term of order $1$}
\noindent
In this section, we want to prove that:
\begin{align*}
2d(H)\int_a^b\bigg(\int_t^{t+\epsilon}D^*_{\delta_s\circ I^{\frac{H}{2}}_+}(\nabla^{\frac{H}{2}}(F(X^H_t))(s,.))\frac{ds}{\epsilon}\bigg)dt \underset{\epsilon\rightarrow 0^+}{\overset{(L^2)}{\longrightarrow}} B(H)\nabla^*\bigg(\int_a^b F'(X^H_t)(t-.)^{\frac{H}{2}-1}_+I_1(l^H_{t,t})dt\bigg).
\end{align*}
For this purpose, we proceed in the following way:
\begin{itemize}
\item First, we compute explicitly $\nabla^{\frac{H}{2}}(F(X^H_t))(s,.)$
\item Then, we prove that $\int_t^{t+\epsilon}D^*_{\delta_s\circ I^{\frac{H}{2}}_+}(\nabla^{\frac{H}{2}}(F(X^H_t))(s,.))\frac{ds}{\epsilon}$ admits a representation from which it is clearly in $(L^2)$
\item Finally, we prove the convergence result
\end{itemize}
First of all, we introduce a sequence of kernels $\big(K^{k}_t(s,r)\big)_k$ defined by
\begin{align*}
&\forall s,r\in(t,+\infty),\ K_t^0(s,r)=|s-r|^{H-1},\ K_t^1(s,r)=\int_{0}^t|s-u|^{H-1}|r-u|^{H-1}du\\
&\forall k\geq 3, K^{k-2}_t(s,r)=\int_0^t...\int_0^t|s-x_1|^{H-1}|x_2-x_1|^{H-1}...|x_{k-2}-x_{k-3}|^{H-1}|r-x_{k-2}|^{H-1}dx_1...dx_{k-2}.
\end{align*}

\begin{rem}\label{ker}
\begin{itemize}
\item We note that for any $t\in (a,b)$, we have:
\begin{align*}
K_t^1(t,t)&=\int_{0}^t|t-u|^{H-1}|t-u|^{H-1}du=\dfrac{t^{2H-1}}{2H-1}<\infty,\\
K_t^2(t,t)&\leq \bigg(\int_0^t\int_0^t|x_1-x_2|^{2H-2}dx_1dx_2\bigg)^{\frac{1}{2}}\bigg(\int_0^t\int_0^t|x_1-t|^{2H-2}|x_2-t|^{2H-2}dx_1dx_2\bigg)^{\frac{1}{2}},\\
&\leq \bigg(\dfrac{t^{2H}}{H(2H-1)}\bigg)^{\frac{1}{2}}\dfrac{t^{2H-1}}{2H-1}=\dfrac{t^{3H-1}}{\sqrt{H(2H-1)}(2H-1)}<\infty.
\end{align*}
\item Moreover, by Lebesgue dominated convergence theorem, one can show that $K^1_t(.,.)$ and $K^2_t(.,.)$ are continuous on $[t,+\infty)\times [t,+\infty)$.
\item Finally, for any $s\in [t,+\infty)$, we have:
\begin{align*}
K^2_t(s,s)\leq K^2_t(t,t).
\end{align*}
\end{itemize}
\end{rem}
\noindent
Then, we have the following technical lemma.

\begin{lem}\label{tech1}
We have:
\begin{align*}
\nabla^{\frac{H}{2}}(F(X^H_t))&(s,\omega)=\dfrac{2}{\Gamma(\frac{H}{2})}\sqrt{\dfrac{H(2H-1)}{2}}I_1(\int_0^t(u-.)^{\frac{H}{2}-1}_+|s-u|^{H-1}du)F'(X^H_t).\\
\end{align*}
\end{lem}

\begin{proof}
By Proposition \ref{ProdChain} and by definition of the operator $\nabla^{\frac{H}{2}}$, we have:
\begin{align*}
\nabla^{\frac{H}{2}}(F(X^H_t))(s,\omega)=\nabla^{\frac{H}{2}}(X^H_t)&(s,\omega)F'(X^H_t)(\omega).
\end{align*}
Moreover, by Proposition \ref{ExGrad1} with $\alpha=H/2$, we obtain:
\begin{align*}
&\nabla^{\frac{H}{2}}(F(X^H_t))(s,\omega)=2I_1(I^{(0,\frac{H}{2})}_{++}(f^H_t)(.,s))F'(X^H_t)(\omega),\\
&\nabla^{\frac{H}{2}}(F(X^H_t))(s,\omega)=\dfrac{2d(H)}{(\Gamma(\frac{H}{2}))^3}I_1\bigg(\int_0^t(u-.)^{\frac{H}{2}-1}_{+}\big(\int_{\mathbb{R}}(s-x)^{\frac{H}{2}-1}_+(u-x)^{\frac{H}{2}-1}_+dx\big)du\bigg)F'(X^H_t)(\omega),\\
&\nabla^{\frac{H}{2}}(F(X^H_t))(s,\omega)=\dfrac{2d(H)\beta(1-H,\frac{H}{2})}{(\Gamma(\frac{H}{2}))^3}I_1\bigg(\int_0^t(u-.)^{\frac{H}{2}-1}_{+}|u-s|^{H-1}du\bigg)F'(X^H_t)(\omega),\\
&\nabla^{\frac{H}{2}}(F(X^H_t))(s,\omega)=\dfrac{2}{\Gamma(\frac{H}{2})}\sqrt{\dfrac{H(2H-1)}{2}}I_1(\int_0^t(u-.)^{\frac{H}{2}-1}_+|s-u|^{H-1}du)F'(X^H_t)(\omega).\
\end{align*}
\end{proof}
\noindent
Then, we obtain the following representation.

\begin{prop}\label{RepRos1}
Let $t\in (a,b)$ and $\epsilon>0$ such that $t+\epsilon<b$. We have:
\begin{align*}
\int_{t}^{t+\epsilon}D^*_{\delta_s\circ I^{\frac{H}{2}}_+}(\nabla^{\frac{H}{2}}(F(X^H_t))(s,))\frac{ds}{\epsilon}\overset{(L^2)}{=}\frac{1}{\epsilon}(\nabla^{\frac{H}{2}})^*(\mathbb{I}_{(t,t+\epsilon)}\nabla^{\frac{H}{2}}(F(X^H_t))).
\end{align*}
\end{prop}

\begin{proof}
First of all, we have:
\begin{align*}
\nabla^{\frac{H}{2}}(F(X^H_t))(s,)=\frac{2}{\Gamma(\frac{H}{2})}\sqrt{\frac{H(2H-1)}{2}}I_1(\int_0^t(u-.)^{\frac{H}{2}-1}_+|s-u|^{H-1}du)F'(X^{H}_t).
\end{align*}
Thus,
\begin{align*}
\int_{t}^{t+\epsilon}D^*_{\delta_s\circ I^{\frac{H}{2}}_+}(\nabla^{\frac{H}{2}}(F(X^H_t))(s,))\frac{ds}{\epsilon}=\frac{2}{\Gamma(\frac{H}{2})}\sqrt{\frac{H(2H-1)}{2}}\int_a^b\mathbb{I}_{(t,t+\epsilon)}(s)D^*_{\delta_s\circ I^{\frac{H}{2}}_+}(I_1(l^H_{s,t})F'(X^{H}_t))\frac{ds}{\epsilon},
\end{align*}
where $l^H_{s,t}(.)=\int_0^t(u-.)^{H/2-1}_+|s-u|^{H-1}du$. We want to apply Proposition \ref{RepAdGrad1} in order to obtain the result. For this purpose, we have to prove that:
\begin{align*}
\int_t^{t+\epsilon}\|I_1(l^H_{s,t})F'(X^H_t)\|^2_{(\mathcal{W}^{1,2})}ds<+\infty.
\end{align*}
By Proposition $1.5.6$ of \cite{N06}, we have:
\begin{align*}
\|I_1(l^H_{s,t})F'(X^H_t)\|_{(\mathcal{W}^{1,2})}\leq C\|I_1(l^H_{s,t})\|_{(\mathcal{W}^{1,4})}\|F'(X^H_t)\|_{(\mathcal{W}^{1,4})},
\end{align*}
Let us estimate $\|I_1(l^H_{s,t})\|_{(\mathcal{W}^{1,4})}$. We have:
\begin{align*}
\|I_1(l^H_{s,t})\|^4_{(\mathcal{W}^{1,4})}=\mathbb{E}[(I_1(l^H_{s,t}))^4]+\mathbb{E}[\|\nabla(I_1(l^H_{s,t}))\|^4_{L^2(\mathbb{R})}]
\end{align*}
Since $I_1(l^H_{s,t})$ is a Gaussian random variable, we have:
\begin{align*}
\mathbb{E}[(I_1(l^H_{s,t}))^4]&=3\mathbb{E}[(I_1(l^H_{s,t}))^2]^2,\\
&=3(\beta(1-H,\frac{H}{2})\int_0^t\int_0^t|s-u|^{H-1}|s-v|^{H-1}|u-v|^{H-1}dudv)^2,\\
&=3(\beta(1-H,\frac{H}{2})K^2_t(s,s))^2.
\end{align*}
Moreover, we have,
\begin{align*}
\|\nabla(I_1(l^H_{s,t}))\|_{L^2(\mathbb{R})}=\|l^H_{s,t}\|_{L^2(\mathbb{R})}=(\beta(1-H,\frac{H}{2})K^2_t(s,s))^{\frac{1}{2}}.
\end{align*}
Thus,
\begin{align*}
\mathbb{E}[\|\nabla(I_1(l^H_{s,t}))\|^4_{L^2(\mathbb{R})}]=\beta(1-H,\frac{H}{2})^2K^2_t(s,s)^2.
\end{align*}
Therefore, we obtain the following equality:
\begin{align*}
&\|I_1(l^H_{s,t})\|_{(\mathcal{W}^{1,4})}=\left(4\beta(1-H,\frac{H}{2})^2(K^2_t(s,s))^2\right)^{\frac{1}{4}},\\
&\|I_1(l^H_{s,t})\|_{(\mathcal{W}^{1,4})}=\sqrt{2\beta(1-H,\frac{H}{2})K^2_t(s,s)}.
\end{align*}
Moreover, since $F$ is an infinitely continuously differentiable function on $\mathbb{R}$ such that $F$ and its derivatives have polynomial growth, $F'(X^H_t)\in (\mathcal{W}^{\infty,\infty})$. Thus, $\|F'(X^H_t)\|_{(\mathcal{W}^{1,4})}$ is finite and independant of $s$. Finally, the kernel $K^2_t$ is continuous on $[t,+\infty)\times[t,+\infty)$. Therefore,
\begin{align*}
\int_{t}^{t+\epsilon}\|I_1(l^H_{s,t})F'(X^H_t)\|^2_{(\mathcal{W}^{1,2})}ds\leq 2C\beta(1-H,\frac{H}{2})\|F'(X^H_t)\|^2_{(\mathcal{W}^{1,4})}\int_{t}^{t+\epsilon}K^2_t(s,s)ds<+\infty.
\end{align*}
This concludes the proof.
\end{proof}
\noindent
A direct application of Proposition \ref{ExAdGrad1} leads to the following equality:
\begin{align*}
\frac{1}{\epsilon}(\nabla^{\frac{H}{2}})^*(\mathbb{I}_{(t,t+\epsilon)}\nabla^{\frac{H}{2}}(F(X^H_t)))=\sqrt{\frac{H(2H-1)}{2}}\frac{2}{\Gamma(\frac{H}{2})^2}\nabla^*\left(F'(X^H_t)\int_t^{t+\epsilon}(r-s)^{\frac{H}{2}-1}_+I_1(l^H_{r,t})\frac{dr}{\epsilon}\right).
\end{align*}

\begin{prop}\label{TraceConvRos}
We have in $(L^2)$:
\begin{align*}
\nabla^{*}\left(\int_a^bF'(X^H_t)\bigg(\int_t^{t+\epsilon}(r-.)_+^{\frac{H}{2}-1}I_1(l^H_{r,t})\frac{dr}{\epsilon}\bigg)dt\right)\underset{\epsilon\rightarrow 0^+}{\longrightarrow}\nabla^{*}\left(\int_a^bF'(X^H_t)(t-.)_+^{\frac{H}{2}-1}I_1(l^H_{t,t})dt\right).
\end{align*}
\end{prop}

\begin{proof}
 In order to prove the proposition, since the operator $\nabla^*$ is continuous from $L^2(\mathbb{R})\otimes (\mathcal{W}^{1,2})$ to $(L^2)$, we will show that, in $L^2(\mathbb{R})\otimes (\mathcal{W}^{1,2})$:
\begin{align*}
\int_a^bF'(X^H_t)\bigg(\int_t^{t+\epsilon}(r-.)_+^{\frac{H}{2}-1}I_1(l^H_{r,t})\frac{dr}{\epsilon}\bigg)dt\underset{\epsilon\rightarrow 0^+}{\longrightarrow}\int_a^bF'(X^H_t)(t-.)_+^{\frac{H}{2}-1}I_1(l^H_{t,t})dt.
\end{align*}
For this purpose, we denote them respectively by $I_{\epsilon}(.)$ and $I(.)$. We have:
\begin{align*}
\int_{\mathbb{R}}\|I_\epsilon(x)-I(x)\|^2_{(\mathcal{W}^{1,2})}dx=&\int_{\mathbb{R}}\int_{(a,b)\times(a,b)}\int_{t_1}^{t_1+\epsilon}\int_{t_2}^{t_2+\epsilon}\langle F'(X^H_{t_1})[(r-x)^{\frac{H}{2}-1}_+I_1(l^H_{r,t_1})-(t_1-x)^{\frac{H}{2}-1}_+I_1(l^H_{t_1,t_1})];\\
&F'(X^H_{t_2})[(s-x)^{\frac{H}{2}-1}_+I_1(l^H_{s,t_2})-(t_2-x)^{\frac{H}{2}-1}_+I_1(l^H_{t_2,t_2})]\rangle_{(\mathcal{W}^{1,2})}\frac{dsdr}{\epsilon^2}dt_1dt_2dx.
\end{align*}
Integrating with respect to $x$, we obtain:
\begin{align*}
\int_{\mathbb{R}}\|I_\epsilon(x)-I(x)\|^2_{(\mathcal{W}^{1,2})}dx=&\beta(1-H,\frac{H}{2})\int_a^b\int_a^b\int_{t_1}^{t_1+\epsilon}\int_{t_2}^{t_2+\epsilon}\bigg(|r-s|^{H-1}\langle F'(X^H_{t_1})I_1(l^H_{r,t_1});F'(X^H_{t_2})I_1(l^H_{s,t_2})\rangle_{(\mathcal{W}^{1,2})}\\
&-|r-t_2|^{H-1}\langle F'(X^H_{t_1})I_1(l^H_{r,t_1});F'(X^H_{t_2})I_1(l^H_{t_2,t_2})\rangle_{(\mathcal{W}^{1,2})}\\
&-|t_1-s|^{H-1}\langle F'(X^H_{t_1})I_1(l^H_{t_1,t_1});F'(X^H_{t_2})I_1(l^H_{s,t_2})\rangle_{(\mathcal{W}^{1,2})}\\
&+|t_1-t_2|^{H-1}\langle F'(X^H_{t_1})I_1(l^H_{t_1,t_1});F'(X^H_{t_2})I_1(l^H_{t_2,t_2})\rangle_{(\mathcal{W}^{1,2})}\bigg)\frac{dsdr}{\epsilon^2}dt_1dt_2.
\end{align*}
To prove the result, we want to apply Lebesgue dominated convergence theorem. For this purpose, we have to show, at first, that for almost every $(t_1,t_2)\in (a,b)\times (a,b)$:
\begin{align*}
&\int_{t_1}^{t_1+\epsilon}\int_{t_2}^{t_2+\epsilon}\bigg(|r-s|^{H-1}\langle F'(X^H_{t_1})I_1(l^H_{r,t_1});F'(X^H_{t_2})I_1(l^H_{s,t_2})\rangle_{(\mathcal{W}^{1,2})}\\
&-|r-t_2|^{H-1}\langle F'(X^H_{t_1})I_1(l^H_{r,t_1});F'(X^H_{t_2})I_1(l^H_{t_2,t_2})\rangle_{(\mathcal{W}^{1,2})}\\
&-|t_1-s|^{H-1}\langle F'(X^H_{t_1})I_1(l^H_{t_1,t_1});F'(X^H_{t_2})I_1(l^H_{s,t_2})\rangle_{(\mathcal{W}^{1,2})}\\
&+|t_1-t_2|^{H-1}\langle F'(X^H_{t_1})I_1(l^H_{t_1,t_1});F'(X^H_{t_2})I_1(l^H_{t_2,t_2})\rangle_{(\mathcal{W}^{1,2})}\bigg)\frac{dsdr}{\epsilon^2}\underset{\epsilon\rightarrow 0^+}{\longrightarrow} 0
\end{align*}
In order to do so, we will show that $(r,s)\rightarrow \langle F'(X^H_{t_1})I_1(l^H_{r,t_1});F'(X^H_{t_2})I_1(l^H_{s,t_2})\rangle_{(\mathcal{W}^{1,2})}$ is continuous in $(t_1^+,t_2^+)$. Let $(r,s)\in[t_1,+\infty)\times[t_2,+\infty)$, we have:
\begin{align*}
(I)&=|\langle F'(X^H_{t_1})I_1(l^H_{r,t_1});F'(X^H_{t_2})I_1(l^H_{s,t_2})\rangle_{1,2}-\langle F'(X^H_{t_1})I_1(l^H_{t_1,t_1});F'(X^H_{t_2})I_1(l^H_{t_2,t_2})\rangle_{1,2}|,\\
(I)&\leq |\langle F'(X^H_{t_1})[I_1(l^H_{r,t_1})-I_1(l^H_{t_1,t_1})];F'(X^H_{t_2})I_1(l^H_{s,t_2})\rangle_{1,2}|+|\langle F'(X^H_{t_1})I_1(l^H_{t_1,t_1});F'(X^H_{t_2})[I_1(l^H_{s,t_2})-I_1(l^H_{t_2,t_2})]\rangle_{1,2}|,\\
(I)&\leq \|F'(X^H_{t_1})[I_1(l^H_{r,t_1})-I_1(l^H_{t_1,t_1})]\|_{1,2}\|F'(X^H_{t_2})I_1(l^H_{s,t_2})\|_{1,2}+\|F'(X^H_{t_1})I_1(l^H_{t_1,t_1})\|_{1,2}\\
&\times\|F'(X^H_{t_2})[I_1(l^H_{s,t_2})-I_1(l^H_{t_2,t_2})]\|_{1,2}.
\end{align*}
Using Proposition $1.5.6$ of \cite{N06}, we obtain for the first term on the right-hand side of the previous inequality:
\begin{align*}
\|F'(X^H_{t_1})[I_1(l^H_{r,t_1})-I_1(l^H_{t_1,t_1})]\|_{1,2}\|F'(X^H_{t_2})I_1(l^H_{s,t_2})\|_{1,2}&\leq C\|F'(X^H_{t_1})\|_{1,4}\|F'(X^H_{t_2})\|_{1,4}\|[I_1(l^H_{r,t_1})-I_1(l^H_{t_1,t_1})]\|_{1,4}\\
&\times\|I_1(l^H_{s,t_2})\|_{1,4},\\
\end{align*}
Moreover, as in the proof of the previous proposition and since $K^2_{t_2}(s,s)\leq K_{t_2}^2(t_2,t_2)$ for any $s\in [t_2,+\infty)$, we obtain:
\begin{align*}
\|I_1(l^H_{s,t_2})\|_{1,4}&=\sqrt{2}\|l^H_{s,t_2}\|_{L^2(\mathbb{R})}=\sqrt{2\beta(1-H,\frac{H}{2})K^2_{t_2}(s,s)},\\
&\leq C_H\left(K_{t_2}^2(t_2,t_2)\right)^{\frac{1}{2}}=C_H (t_2)^{\frac{3H-1}{2}}\left(K^2_{1}(1,1)\right)^{\frac{1}{2}}.
\end{align*}
Similarly, we have,
\begin{align*}
\|[I_1(l^H_{r,t_1})-I_1(l^H_{t_1,t_1})]\|_{1,4}=\sqrt{2}\|l^H_{r,t_1}-l^H_{t_1,t_1}\|_{L^2(\mathbb{R})}=\sqrt{2\beta(1-H,\frac{H}{2})[K_{t_1}^2(r,r)+K_{t_1}^2(t_1,t_1)-2K^2_{t_1}(r,t_1)]}.
\end{align*}
Thus, we have:
\begin{align*}
(I)&\leq C_H\|F'(X^H_{t_1})\|_{1,4}\|F'(X^H_{t_2})\|_{1,4}(t_2)^{\frac{3H-1}{2}}\left(K^2_{1}(1,1)\right)^{\frac{1}{2}}\sqrt{[K_{t_1}^2(r,r)+K_{t_1}^2(t_1,t_1)-2K^2_{t_1}(r,t_1)]}\\
&+C_H\|F'(X^H_{t_1})\|_{1,4}\|F'(X^H_{t_2})\|_{1,4}(t_1)^{\frac{3H-1}{2}}\left(K^2_{1}(1,1)\right)^{\frac{1}{2}}\sqrt{[K_{t_2}^2(s,s)+K_{t_2}^2(t_2,t_2)-2K^2_{t_2}(s,t_2)]}.
\end{align*}
Since, the kernel $K_t^2(s,r)$ is continuous on $[t,+\infty)\times [t,+\infty)$, we obtain the continuity of $(r,s)\rightarrow \langle F'(X^H_{t_1})I_1(l^H_{r,t_1});F'(X^H_{t_2})I_1(l^H_{s,t_2})\rangle_{(\mathcal{W}^{1,2})}$ in $(t^+_1,t^+_2)$. Consequently, we obtain that for any $t_1\ne t_2$ in $(a,b)\times (a,b)$:
\begin{align*}
&\int_{t_1}^{t_1+\epsilon}\int_{t_2}^{t_2+\epsilon}\bigg(|r-s|^{H-1}\langle F'(X^H_{t_1})I_1(l^H_{r,t_1});F'(X^H_{t_2})I_1(l^H_{s,t_2})\rangle_{(\mathcal{W}^{1,2})}\\
&-|r-t_2|^{H-1}\langle F'(X^H_{t_1})I_1(l^H_{r,t_1});F'(X^H_{t_2})I_1(l^H_{t_2,t_2})\rangle_{(\mathcal{W}^{1,2})}\\
&-|t_1-s|^{H-1}\langle F'(X^H_{t_1})I_1(l^H_{t_1,t_1});F'(X^H_{t_2})I_1(l^H_{s,t_2})\rangle_{(\mathcal{W}^{1,2})}\\
&+|t_1-t_2|^{H-1}\langle F'(X^H_{t_1})I_1(l^H_{t_1,t_1});F'(X^H_{t_2})I_1(l^H_{t_2,t_2})\rangle_{(\mathcal{W}^{1,2})}\bigg)\frac{dsdr}{\epsilon^2}\underset{\epsilon\rightarrow 0^+}{\longrightarrow} 0
\end{align*}
To conclude, we need to check the dominating condition by bounding the previous integral which we denote by $(II)$. To this end, we have to bound, at first:
\begin{itemize}
\item $m_{r,s}=\langle F'(X^H_{t_1})I_1(l^H_{r,t_1});F'(X^H_{t_2})I_1(l^H_{s,t_2})\rangle_{(\mathcal{W}^{1,2})}$, 
\item $m_{r,t_2}=\langle F'(X^H_{t_1})I_1(l^H_{r,t_1});F'(X^H_{t_2})I_1(l^H_{t_2,t_2})\rangle_{(\mathcal{W}^{1,2})}$, 
\item $m_{t_1,s}=\langle F'(X^H_{t_1})I_1(l^H_{t_1,t_1});F'(X^H_{t_2})I_1(l^H_{s,t_2})\rangle_{(\mathcal{W}^{1,2})}$, 
\item $m_{t_1,t_2}=\langle F'(X^H_{t_1})I_1(l^H_{t_1,t_1});F'(X^H_{t_2})I_1(l^H_{t_2,t_2})\rangle_{(\mathcal{W}^{1,2})}$.
\end{itemize}
Then, we have:
\begin{align*}
&|m_{r,s}|\leq C\|F'(X^H_{t_1})\|_{1,4}\|F'(X^H_{t_2})\|_{1,4}\|I_1(l^H_{r,t_1})\|_{1,4}\|I_1(l^H_{s,t_2})\|_{1,4},\\
&|m_{r,s}|\leq C_H\|F'(X^H_{t_1})\|_{1,4}\|F'(X^H_{t_2})\|_{1,4}\sqrt{K^2_{t_1}(r,r)K^2_{t_2}(s,s)},\\
&|m_{r,s}|\leq C_H\|F'(X^H_{t_1})\|_{1,4}\|F'(X^H_{t_2})\|_{1,4}\sqrt{K^2_{t_1}(t_1,t_1)K^2_{t_2}(t_2,t_2)},\\
&|m_{r,s}|\leq C'_H\|F'(X^H_{t_1})\|_{1,4}\|F'(X^H_{t_2})\|_{1,4}(t_1)^{\frac{3H-1}{2}}(t_2)^{\frac{3H-1}{2}}.
\end{align*}
Moreover, for any $i\in \{1,2\}$, we have, by definition:
\begin{align*}
\|F'(X^H_{t_i})\|_{1,4}=\left(\mathbb{E}[(F'(X^H_{t_i}))^4]+\mathbb{E}[\|\nabla(F'(X^H_{t_i}))\|^4_{L^2(\mathbb{R})}]\right)^{\frac{1}{4}}
\end{align*}
By hypothesis, $F$ and all its derivatives have polynomial growth. Then, using Hypercontractivity,
\begin{align*}
&\mathbb{E}[(F'(X^H_{t_i}))^4]\leq C\bigg(1+4Nt_i^H+C^{4N}_2t_i^{2H}+\sum_{p=3}^{4N}C^{4N}_pt_i^{Hp}(p-1)^p\bigg),\\
&\mathbb{E}[(F'(X^H_{t_i}))^4]\leq C P_{N_1}(t^H_i),
\end{align*}
for some $C>0$, $N_1\geq 1$ and $P_{N_1}(.)$ a polynomial of degree $4N_1$ with strictly positive coefficients. Moreover, we have:
\begin{align*}
\|\nabla(F'(X^H_{t_i}))\|_{L^2(\mathbb{R})}&=2|F^{(2)}(X^H_{t_i})|\left(\int_{\mathbb{R}}(I_1(f^H_{t_i}(x,.)))^2dx\right)^{\frac{1}{2}},\\
&=2|F^{(2)}(X^H_{t_i})|\left(I_2(\int_{\mathbb{R}}f_{t_i}^H(x,.)\hat{\otimes}f^H_{t_i}(x,*)dx)+\|f^H_{t_i}\|^2_{L^2(\mathbb{R}^2)}\right)^{\frac{1}{2}},\\
&=2|F^{(2)}(X^H_{t_i})|\left(I_2(\int_{\mathbb{R}}f_{t_i}^H(x,.)\hat{\otimes}f^H_{t_i}(x,*)dx)+\frac{t_i^{2H}}{2}\right)^{\frac{1}{2}}.
\end{align*}
Thus,
\begin{align*}
\mathbb{E}[\|\nabla(F'(X^H_{t_i}))\|^4_{L^2(\mathbb{R})}]&=16\mathbb{E}\bigg[|F^{(2)}(X^H_{t_i})|^4\left(I_2(\int_{\mathbb{R}}f_{t_i}^H(x,.)\hat{\otimes}f^H_{t_i}(x,*)dx)+\frac{t_i^{2H}}{2}\right)^2\bigg],\\
&\leq 32\mathbb{E}\bigg[|F^{(2)}(X^H_{t_i})|^4\left(I_2(\int_{\mathbb{R}}f_{t_i}^H(x,.)\hat{\otimes}f^H_{t_i}(x,*)dx)\right)^2\bigg]\\
&+8t_i^{4H}\mathbb{E}\bigg[|F^{(2)}(X^H_{t_i})|^4\bigg].
\end{align*}
As previously, we obtain, for some $C>0$, $N_2\geq 1$ and $Q_{N_2}$ a polynomial of degree $4N_2$ with strictly positive coefficients:
\begin{align*}
\mathbb{E}\bigg[|F^{(2)}(X^H_{t_i})|^4\bigg]\leq C Q_{N_2}(t^H_i)
\end{align*}
Moreover, by Cauchy-Schwarz inequality, we have:
\begin{align*}
\mathbb{E}\bigg[|F^{(2)}(X^H_{t_i})|^4\left(I_2(\int_{\mathbb{R}}f_{t_i}^H(x,.)\hat{\otimes}f^H_{t_i}(x,*)dx)\right)^2\bigg]&\leq
\mathbb{E}\bigg[|F^{(2)}(X^H_{t_i})|^8\bigg]^{\frac{1}{2}}\mathbb{E}\bigg[\left(I_2(\int_{\mathbb{R}}f_{t_i}^H(x,.)\hat{\otimes}f^H_{t_i}(x,*)dx)\right)^4\bigg]^{\frac{1}{2}},\\
&\leq C (R_{N_3}(t_i^H))^{\frac{1}{2}}\mathbb{E}\bigg[\left(I_2(\int_{\mathbb{R}}f_{t_i}^H(x,.)\hat{\otimes}f^H_{t_i}(x,*)dx)\right)^4\bigg]^{\frac{1}{2}},\\
&\leq C (R_{N_3}(t_i^H))^{\frac{1}{2}}\mathbb{E}\bigg[\left(I_2(\int_{\mathbb{R}}f_{t_i}^H(x,.)\hat{\otimes}f^H_{t_i}(x,*)dx)\right)^2\bigg],\\
&\leq C(R_{N_3}(t_i^H))^{\frac{1}{2}}\bigg(\int_{\mathbb{R}^2}\bigg(\int_{\mathbb{R}}f_{t_i}^H(y,x_1)f^H_{t_i}(y,x_2)dy\bigg)^2dx_1dx_2\bigg),\\
&\leq C_H(R_{N_3}(t_i^H))^{\frac{1}{2}}\bigg(\int_{\mathbb{R}^2}\bigg(\int_0^{t_i}\int_0^{t_i}(r_2-x_2)^{\frac{H}{2}-1}_+(r_1-x_1)^{\frac{H}{2}-1}_+\\
&\times|r_1-r_2|^{H-1}dr_1dr_2\bigg)^2dx_1dx_2\bigg),\\
&\leq C_H(R_{N_3}(t_i^H))^{\frac{1}{2}}\bigg(\int_0^{t_i}\int_0^{t_i}\int_0^{t_i}\int_0^{t_i}|r_1-r_2|^{H-1}|r_2-r_3|^{H-1}\\
&\times|r_3-r_4|^{H-1}|r_4-r_1|^{H-1}dr_1dr_2dr_3dr_4\bigg),\\
&\leq C_H(R_{N_3}(t_i^H))^{\frac{1}{2}}t_i^{4H},
\end{align*}
where $N_3\geq 1$ and $R_{N_3}$ is a polynomial of degree $8N_3$ with strictly positive coefficients. Consequently, we obtain:
\begin{align*}
|m_{r,s}|\leq C'_H\prod_{i=1}^2(t_i)^{\frac{3H-1}{2}}[P_{N_1}(t^H_i)+t^{4H}_iQ_{N_2}(t^H_i)+t^{4H}_i(R_{N_3}(t_i^H))^{\frac{1}{2}}]^{\frac{1}{4}}.
\end{align*}
Similar bounds hold for $|m_{t_1,s}|$, $|m_{r,t_2}|$ and $|m_{t_1,t_2}|$. Therefore, using the fact that $H-1<0$, we have:
\begin{align*}
|(II)|\leq C_H p(t_1,t_2)|t_1-t_2|^{H-1},
\end{align*}
where $p(t_1,t_2)=\prod_{i=1}^2(t_i)^{(3H-1)/2}[P_{N_1}(t^H_i)+t^{4H}_iQ_{N_2}(t^H_i)+t^{4H}_i(R_{N_3}(t_i^H))^{1/2}]^{1/4}$. Standard computations lead to the following bound:
\begin{align*}
\int_{(a,b)\times(a,b)}p(t_1,t_2)|t_1-t_2|^{H-1}dt_1dt_2\leq C_H\|p\|_{\infty,[a,b]\times[a,b]}(b-a)^{H+1}<+\infty.
\end{align*}
This concludes the proof of the proposition.
\end{proof}

\subsubsection{Trace term of order $2$}
\noindent
The strong convergence of the trace term of order $2$ is easier to handle with. As we will see from the next computations, it admits the following representation:
\begin{align*}
\forall \epsilon>0,\ \int_a^b\langle \nabla^{(2)}\big(F(X^H_t)\big); h^H_{t,t+\epsilon}\rangle_{\hat{L}^2(\mathbb{R}^2)}dt=d(H)\int_a^b\Big(\int_{t}^{t+\epsilon}\nabla^{(2),\frac{H}{2}}(F(X^H_t))(s,s,)\frac{ds}{\epsilon}\Big)dt,
\end{align*}
Then, the first step is to compute $\nabla^{(2),H/2}(F(X^H_t))$ as done in the next technical lemma.

\begin{lem}\label{tech2}
We have:
\begin{align*}
d(H)\nabla^{(2),\frac{H}{2}}(&F(X^H_t))(s_1,s_2,\omega)=H(2H-1)K_t^1(s_1,s_2)F'(X^H_t)+4\big(\sqrt{\dfrac{H(2H-1)}{2}}\big)^3K^2_t(s_1,s_2)F^{(2)}(X^H_t)\\&+\dfrac{2d(H)}{(\Gamma(\frac{H}{2}))^2}H(2H-1)I_2(\int_0^t\int_0^t(u-.)^{\frac{H}{2}-1}_+(v-*)^{\frac{H}{2}-1}_+|u-s_1|^{H-1}|v-s_2|^{H-1}dudv)F^{(2)}(X^H_t).
\end{align*}
\end{lem}

\begin{proof}
Using Proposition \ref{ProdChain} and the definition of the operator $\nabla^{(2),\frac{H}{2}}$, we have:
\begin{align*}
\nabla^{(2),\frac{H}{2}}(F(X^H_t))(s_1,s_2,\omega)=\nabla^{(2),\frac{H}{2}}(X^H_t)(s_1,s_2,\omega)F'(X^H_t)+\nabla^{\frac{H}{2}}(X^H_t)(s_1,\omega)\nabla^{\frac{H}{2}}(X^H_t)(s_2,\omega)F^{(2)}(X^H_t),
\end{align*}
Using Proposition \ref{ExGrad1} and Proposition \ref{ExGrad2}, we have:
\begin{align*}
\nabla^{(2),\frac{H}{2}}(F(X^H_t))(s_1,s_2,\omega)=2I^{(\frac{H}{2},\frac{H}{2})}_{++}(f^H_t)(s_1,s_2)F'(X^H_t)+4I_1(I^{(0,\frac{H}{2})}_{++}(f^H_t)(.,s_1))I_1(I^{(0,\frac{H}{2})}_{++}(f^H_t)(.,s_2))F^{(2)}(X^H_t).
\end{align*}
By the multiplication formula from Malliavin calculus (Proposition $1.1.3$ of \cite{N06}), we get:
\begin{align*}
\nabla^{(2),\frac{H}{2}}(F(X^H_t))&(s_1,s_2,\omega)=2I^{(\frac{H}{2},\frac{H}{2})}_{++}(f^H_t)(s_1,s_2)F'(X^H_t)+4[I_2(I^{(0,\frac{H}{2})}_{++}(f^H_t)(.,s_1)\hat{\otimes}I^{(0,\frac{H}{2})}_{++}(f^H_t)(.,s_2))\\
&+<I^{(0,\frac{H}{2})}_{++}(f^H_t)(.,s_1);I^{(0,\frac{H}{2})}_{++}(f^H_t)(.,s_2)>]F^{(2)}(X^H_t),\\
&=2\dfrac{d(H)}{(\Gamma(\frac{H}{2}))^4}\int_0^t\big(\int_{\mathbb{R}}(s_1-x)^{\frac{H}{2}-1}_+(u-x)^{\frac{H}{2}-1}_+dx\big)\big(\int_{\mathbb{R}}(s_2-x)^{\frac{H}{2}-1}_+(u-x)^{\frac{H}{2}-1}_+dx\big)duF'(X^H_t)\\
&+\dfrac{2}{(\Gamma(\frac{H}{2}))^2}H(2H-1)I_2(\int_0^t\int_0^t(u-.)^{\frac{H}{2}-1}_+(v-*)^{\frac{H}{2}-1}_+|u-s_1|^{H-1}|v-s_2|^{H-1}dudv)F^{(2)}(X^H_t)\\
&+2\dfrac{H(2H-1)}{(\Gamma(\frac{H}{2}))^2}\int_0^t\int_0^t\big(\int_{\mathbb{R}}(u-x)^{\frac{H}{2}-1}_{+}(v-x)^{\frac{H}{2}-1}_+dx\big)|s_1-u||s_2-v|dudvF^{(2)}(X^H_t),\\
&=2\dfrac{d(H)\beta^2(1-H,\frac{H}{2})}{(\Gamma(\frac{H}{2}))^4}\int_0^t|u-s_1|^{H-1}|u-s_2|^{H-1}duF'(X^H_t)\\
&+\dfrac{2}{(\Gamma(\frac{H}{2}))^2}H(2H-1)I_2(\int_0^t\int_0^t(u-.)^{\frac{H}{2}-1}_+(v-*)^{\frac{H}{2}-1}_+|u-s_1|^{H-1}|v-s_2|^{H-1}dudv)F^{(2)}(X^H_t)\\
&+2\dfrac{H(2H-1)\beta(1-H,\frac{H}{2})}{(\Gamma(\frac{H}{2}))^2}\int_0^t\int_0^t|u-v|^{H-1}|s_1-u|^{H-1}|s_2-v|^{H-1}dudvF^{(2)}(X^H_t).
\end{align*}
Thus,
\begin{align*}
d(H)\nabla^{(2),\frac{H}{2}}&(F(X^H_t))(s_1,s_2,\omega)=H(2H-1)K^1_t(s_1,s_2)F'(X^H_t)+4\big(\sqrt{\dfrac{H(2H-1)}{2}}\big)^3K^2_t(s_1,s_2)F^{(2)}(X^H_t)\\&+\dfrac{2d(H)}{(\Gamma(\frac{H}{2}))^2}H(2H-1)I_2(\int_0^t\int_0^t(u-.)^{\frac{H}{2}-1}_+(v-*)^{\frac{H}{2}-1}_+|u-s_1|^{H-1}|v-s_2|^{H-1}dudv)F^{(2)}(X^H_t).\
\end{align*}
\end{proof}
\noindent
Then, we consider separately the strong convergence of the appropriate terms coming from the previous decomposition.

\begin{lem}\label{Trace2ConvRos1}
We have  $a.s.$ and in $(L^2)$:
\begin{align*}
&H(2H-1)\int_a^b\big(\int_t^{t+\epsilon}K_t^1(s,s)\dfrac{ds}{\epsilon}\big)F'(X^H_t)dt\underset{\epsilon\rightarrow 0^+}{\longrightarrow}H\int_a^bt^{2H-1}F'(X^H_t)dt,\\
&4\big(\sqrt{\dfrac{H(2H-1)}{2}}\big)^3\int_a^b\big(\int_t^{t+\epsilon}K^2_t(s,s)\dfrac{ds}{\epsilon}\big)F^{(2)}(X^H_t)dt\underset{\epsilon\rightarrow 0^+}{\longrightarrow}\dfrac{H}{2}\kappa_3(X^H_1)\int_a^bt^{3H-1}F^{(2)}(X^H_t)dt.
\end{align*}
\end{lem}

\begin{proof}
By Remark \ref{ker} and Lebesgue dominated convergence theorem, we have:
\begin{align*}
&H(2H-1)\int_a^b\big(\int_t^{t+\epsilon}K_t^1(s,s)\dfrac{ds}{\epsilon}\big)F'(X^H_t)dt\underset{\epsilon\rightarrow 0^+}{\longrightarrow}H(2H-1)\int_a^bK_t^1(t,t)F'(X^H_t)dt,\\
&4\big(\sqrt{\dfrac{H(2H-1)}{2}}\big)^3\int_a^b\big(\int_t^{t+\epsilon}K^2_t(s,s)\dfrac{ds}{\epsilon}\big)F^{(2)}(X^H_t)dt\underset{\epsilon\rightarrow 0^+}{\longrightarrow}4\big(\sqrt{\dfrac{H(2H-1)}{2}}\big)^3\int_a^bK^2_t(t,t)F^{(2)}(X^H_t)dt.
\end{align*}
By scaling property of the kernels, we have:
\begin{align*}
H(2H-1)\int_a^bK_t^1(t,t)F'(X^H_t)dt&=H(2H-1)\int_0^1(1-x)^{2H-2}dx\int_a^bt^{2H-1}F'(X^H_t)dt,\\
4\big(\sqrt{\dfrac{H(2H-1)}{2}}\big)^3\int_a^bK^2_t(t,t)F^{(2)}(X^H_t)dt&=4\big(\sqrt{\dfrac{H(2H-1)}{2}}\big)^3K^2_1(1,1)\int_a^bt^{3H-1}F^{(2)}(X^H_t)dt.
\end{align*}
Moreover, we note that:
\begin{align*}
K^2_1(1,1)=H\int_0^1\int_0^1\int_0^1|x_1-x_2|^{H-1}|x_2-x_3|^{H-1}|x_3-x_1|^{H-1}dx_1dx_2dx_3.
\end{align*}
Thus,
\begin{align*}
&H(2H-1)\int_a^bK_t^1(t,t)F'(X^H_t)dt=H\int_a^bt^{2H-1}F'(X^H_t)dt,\\
&4\big(\sqrt{\dfrac{H(2H-1)}{2}}\big)^3\int_a^bK^2_t(t,t)F^{(2)}(X^H_t)dt=\dfrac{H}{2}\kappa_3(X^H_1)\int_a^bt^{3H-1}F^{(2)}(X^H_t)dt.\
\end{align*}
\end{proof}
\noindent
Finally, we conclude by two lemmas regarding the strong convergence of the last term of the decomposition of $\nabla^{(2),H/2}(F(X^H_t))$ from Lemma \ref{tech2}.
 
\begin{lem}
Let $t\in (a,b)$. We have:
\begin{align*}
\int_{t}^{t+\epsilon}I_2(\int_0^t\int_0^t(u-.)^{\frac{H}{2}-1}_+(v-*)^{\frac{H}{2}-1}_+|u-s|^{H-1}|v-s|^{H-1}dudv)\dfrac{ds}{\epsilon}\\
\overset{(L^2)}{\underset{\epsilon\rightarrow 0^+}{\longrightarrow}}I_2(\int_0^t\int_0^t(u-.)^{\frac{H}{2}-1}_+(v-*)^{\frac{H}{2}-1}_+|u-t|^{H-1}|v-t|^{H-1}dudv).
\end{align*}
\end{lem}

\begin{proof}
First of all, note that:
\begin{align*}
\int_0^t\int_0^t(u-.)^{\frac{H}{2}-1}_+(v-*)^{\frac{H}{2}-1}_+|u-t|^{H-1}|v-t|^{H-1}dudv\in L^{2}(\mathbb{R}^2).
\end{align*}
Indeed, we have:
\begin{align*}
\int_0^t\int_0^t\int_0^t\int_0^t|u-t|^{H-1}|v-t|^{H-1}|u'-t|^{H-1}|v'-t|^{H-1}|u-u'|^{H-1}|v-v'|^{H-1}dudvdu'dv'=(K^2_t(t,t))^2<\infty.
\end{align*}
Let $\epsilon>0$. We have:
\begin{align*}
\mathbb{E}[\bigg(\int_t^{t+\epsilon}I_2(e_{s,t}^H(.,*))\dfrac{ds}{\epsilon}-&I_2(e^H_{t,t}(.,*))\bigg)^2]=\mathbb{E}[\bigg(\int_t^{t+\epsilon}(I_2(e_{s,t}^H(.,*))-I_2(e^H_{t,t}(.,*)))\dfrac{ds}{\epsilon}\bigg)^2],\\
&=\dfrac{2}{\epsilon^2}\int_t^{t+\epsilon}\int_t^{t+\epsilon}<e_{s,t}^H(.,*)-e^H_{t,t}(.,*);e_{s',t}^H(.,*)-e^H_{t,t}(.,*)>_{L^2(\mathbb{R}^2)}dsds',
\end{align*}
where $e_{s,t}^H(.,*)=\int_0^t\int_0^t(u-.)^{\frac{H}{2}-1}_+(v-*)^{\frac{H}{2}-1}_+|u-s|^{H-1}|v-s|^{H-1}dudv$. Moreover, we have:
\begin{align*}
\dfrac{1}{\epsilon^2}\int_t^{t+\epsilon}\int_t^{t+\epsilon}<e_{s,t}^H(.,*);e_{s',t}^H(.,*)>_{L^2(\mathbb{R}^2)}dsds'=\dfrac{(\beta(1-H,\frac{H}{2}))^2}{\epsilon^2}\int_t^{t+\epsilon}\int_t^{t+\epsilon}(K^2_t(s,s'))^2dsds'.
\end{align*}
Thus,
\begin{align*}
\mathbb{E}[\bigg(\int_t^{t+\epsilon}I_2(e_{s,t}^H(.,*))&\dfrac{ds}{\epsilon}-I_2(e^H_{t,t}(.,*))\bigg)^2]=\dfrac{2(\beta(1-H,\frac{H}{2}))^2}{\epsilon^2}\int_t^{t+\epsilon}\int_t^{t+\epsilon}(K^2_t(s,s'))^2dsds'\\
&-\dfrac{4(\beta(1-H,\frac{H}{2}))^2}{\epsilon}\int_t^{t+\epsilon}(K^2_t(t,s))^2ds+2(\beta(1-H,\frac{H}{2}))^2(K^2_t(t,t))^2.
\end{align*}
The continuity of $K^2_t(.,.)$ on $[t,+\infty)\times[t,+\infty)$ concludes the proof. 
\end{proof}

\begin{lem}\label{Trace2ConvRos2}
We have:
\begin{align*}
\int_a^b\big(\int_t^{t+\epsilon}I_2(e_{s,t}^H(.,*))\dfrac{ds}{\epsilon}\big)F^{(2)}(X^H_t)dt\overset{(L^1)}{\underset{\epsilon\rightarrow 0^+}{\longrightarrow}}\int_a^bI_2(e_{t,t}^H(.,*))F^{(2)}(X^H_t)dt.
\end{align*}
\end{lem}

\begin{proof}
This follows from standard estimates, the previous Lemma, properties of the kernel $K_t^2(.,.)$ on $[t,+\infty)\times[t,+\infty)$ and the Lebesgue dominated convergence theorem.
\end{proof}

\section*{Acknowledgements}
 
Part of this work was done when the author was working at MICS laboratory and at Ceremade. The author acknowledges these institutions for providing suitable environment of research.

\def\cprime{$'$}

\end{document}